\documentclass[12pt]{amsart}
\pdfoutput=1

\title{The Erd\H{o}s-Pósa property for infinite graphs}
\author{Thilo Krill}
\address{Universit\"at Hamburg, Department of Mathematics, Bundesstrasse 55 (Geomatikum), 20146 Hamburg, Germany}
\email{thilo.krill@uni-hamburg.de}

\usepackage{amsmath,amssymb,amsthm}
\usepackage{mathtools}
\usepackage{tikz}
\usepackage{comment}
\usepackage{enumitem}
\usepackage{thmtools}
\usepackage{thm-restate}

\usepackage{xcolor}
\usepackage[unicode]{hyperref}
\hypersetup{
	colorlinks,
    linkcolor={red!60!black},
    citecolor={green!60!black},
    urlcolor={blue!60!black},
}
\usepackage[abbrev, msc-links]{amsrefs}

\usepackage[utf8]{inputenc}
\usepackage[T1]{fontenc}
\usepackage{lmodern}
\usepackage[babel]{microtype}
\usepackage[english]{babel}

\usepackage{geometry}
\usepackage{enumitem}

\let\polishlcross=\l
\def\l{\ifmmode\ell\else\polishlcross\fi}

\let\theta=\vartheta
\let\rho=\varrho
\let\phi=\varphi

\def\NN{\mathbb N}

\def\cC{{\mathcal C}}

\def\cF{{\mathcal F}}
\def\cG{{\mathcal G}}
\def\cH{{\mathcal H}}

\def\cM{{\mathcal M}}

\def\cP{{\mathcal P}}

\def\cS{{\mathcal S}}
\def\cT{{\mathcal T}}
\def\cU{{\mathcal U}}
\def\cV{{\mathcal V}}
\def\cW{{\mathcal W}}

\theoremstyle{plain}
\newtheorem{thm}{Theorem}[section]

\newtheorem{prop}[thm]{Proposition}

\newtheorem{lemma}[thm]{Lemma}

\newtheorem{problem}[thm]{Problem}
\newtheorem*{thm*}{Theorem}
\newtheorem*{problem*}{Problem}

\theoremstyle{definition}

\newtheorem{defn}[thm]{Definition}

\newtheorem{case}{Case}
\newtheorem*{defn*}{Definition}

\newcommand{\overcirc}[1]{\overset{\circ}{#1}}
\DeclareMathOperator{\successor}{succ}
\DeclareMathOperator{\br}{br}
\DeclareMathOperator{\cf}{cf}

\begin{document}

\begin{abstract}
We investigate which classes of infinite graphs have the Erd\H{o}s-Pósa property (EPP).
In addition to the usual EPP, we also consider the following infinite variant of the EPP: a class $\mathcal{G}$ of graphs has the $\kappa$-EPP, where $\kappa$ is an infinite cardinal, if for any graph $\Gamma$ there are either $\kappa$ disjoint graphs from $\mathcal{G}$ in $\Gamma$ or there is a set $X$ of vertices of $\Gamma$ of size less than $\kappa$ such that $\Gamma - X$ contains no graph from $\mathcal{G}$.

In particular, we study the ($\kappa$-)EPP for classes consisting of a single infinite graph $G$. We obtain positive results when the set of induced subgraphs of $G$ is labelled well-quasi-ordered, and negative results when $G$ is not a proper subgraph of itself (both results require some additional conditions).
As a corollary, we obtain that every graph which does not contain a path of length $n$ for some $n \in \mathbb{N}$ has the EPP and the $\kappa$-EPP.
Furthermore, we show that the class of all subdivisions of any tree $T$ has the $\kappa$-EPP for every uncountable cardinal $\kappa$, and if $T$ is rayless, also the $\aleph_0$-EPP and the EPP.
\end{abstract}

\maketitle

\section{Introduction}
\label{sec:introduction}

\subsection{The Erd\H{o}s-Pósa property}

Erd\H{o}s and Pósa \cite{erdosposa1965} established the following landmark result:

\begin{thm}\label{thm:original_Erdos-Posa}
There is a function $f: \NN \to \NN$ such that for every graph $\Gamma$ and every $k \in \NN$ one of the following holds:
\begin{itemize}
\item $\Gamma$ contains $k$ disjoint cycles, or
\item there is set $X \subseteq V(\Gamma)$ of size at most $f(k)$ such that $\Gamma - X$ is a forest.
\end{itemize}
\end{thm}

The following definition makes it possible to investigate whether theorems of this type hold for graphs other than cycles:

\begin{defn}\label{def:EPP}
A class $\cG$ of graphs has the \emph{Erd\H{o}s-Pósa property (EPP)} if there is a function $f:\NN\to\NN$ such that for every graph $\Gamma$ (referred to as host graph) and every $k \in \NN$ one of the following holds:
\begin{itemize}
\item $\Gamma$ contains $k$ disjoint copies of graphs from $\cG$ as subgraphs, or
\item there is set $X \subseteq V(\Gamma)$ of size at most $f(k)$ such that $\Gamma - X$ does not contain any copy of a graph from $\cG$ as a subgraph.
\end{itemize}
\end{defn}

For an overview of known results on classes of graphs $\cG$ which have or do not have the EPP, see \cite{raymond2017recent}.
The graphs in $\cG$ and the graph $\Gamma$ in Definition \ref{def:EPP} are usually required to be finite, but here we allow them to be infinite.
A first natural question would be whether known results for finite graphs
(such as Theorem \ref{thm:original_Erdos-Posa})
still hold if the host graph $\Gamma$ can be infinite.
In fact, this follows from a simple compactness argument (Proposition \ref{prop:finite_EPP}).

However, a completely new type of problem arises when not only $\Gamma$ but also the graphs in $\cG$ are infinite.
The author is not aware of any previous work on this topic.
In this paper, we focus mainly on the special case where $\cG$ consists of a single infinite graph.
If $\cG = \{ G \}$ and $\cG$ has the EPP, then we also say that $G$ has the \emph{EPP}.

\begin{problem}
Which infinite graphs have the EPP?
\end{problem}

This problem may seem surprising, considering that every finite graph trivially has the EPP (Proposition \ref{prop:finite_graphs_have_EPP}).
However, not every infinite graph has the EPP; the simplest counterexample is the ray (Proposition \ref{prop:ray_does_not_have_EPP}).
More counterexamples are given in Proposition \ref{prop:finitely_many_rays}.
On the other hand, we prove that rayless graphs with a certain additional property have the EPP:
we say that a class $\cC$ of graphs is \emph{labelled well-quasi-ordered (lwqo)} if, for any finite set $L$, the class of all graphs from $\cC$ with vertices labelled by elements of $L$ is well-quasi-ordered (wqo) by label-preserving subgraph embeddings.

\begin{restatable}{thm}{introEPP}\label{thm:intro_EPP}
Let $G$ be any rayless graph such that the set of induced subgraphs of $G$ is lwqo. Then $G$ has the EPP.
\end{restatable}

An example of a class of graphs that is lwqo is the class of all graphs that exclude a path of fixed length as a subgraph.
This was shown for finite graphs by Ding \cite{ding1992subgraphs} and generalised to infinite graphs by Jia \cite{jia2015excluding}.
Hence we obtain the following corollary of Theorem \ref{thm:intro_EPP}:

\begin{restatable}{cor}{excludedpathEPP}\label{cor:excluded_path_EPP}
Let $G$ be a graph that does not contain a path of length $n$ as a subgraph for some $n \in \NN$. Then $G$ has the EPP.
\end{restatable}

For more results on labelled well-quasi-ordering see \cites{atminas2015labelled, brignall2018counterexample, daligault1990well, korpelainen2011two, pouzet1972un}.
We believe that extending lwqo results to infinite graphs is an interesting problem and will lead to further applications of Theorem \ref{thm:intro_EPP}.

\subsection{The $\kappa$-Erd\H{o}s-Pósa property}

We also consider the following infinite version of the Erd\H{o}s-Pósa property, which will turn out to behave similarly to the ``normal'' Erd\H{o}s-Pósa property:

\begin{defn}
Let $\kappa$ be any infinite cardinal. We say that a class $\cG$ of graphs has the \emph{$\kappa$-Erd\H{o}s-Pósa property ($\kappa$-EPP)} if for every graph $\Gamma$ one of the following holds:
\begin{itemize}
\item $\Gamma$ contains $\kappa$ disjoint copies of graphs from $\cG$ as subgraphs, or
\item there is set $X \subseteq V(\Gamma)$ of size less than $\kappa$ such that $\Gamma - X$ does not contain any copy of a graph from $\cG$ as a subgraph.
\end{itemize}
\end{defn}

Again, the author is not aware of any previous work on this topic.
We say that a graph $G$ has the EPP if the class $\{ G \}$ does.

\begin{problem}
Which infinite graphs have the $\kappa$-EPP?
\end{problem}

Addressing this problem, we offer the following results:

\begin{restatable}{thm}{introalephnaughtEPP}\label{thm:intro_aleph_0-EPP}
Let $G$ be any rayless graph such that the set of induced subgraphs of $G$ is lwqo. Then $G$ has the $\aleph_0$-EPP.
\end{restatable}

\begin{restatable}{thm}{introkappaEPP}\label{thm:intro_kappa-EPP}
Let $G$ be any graph admitting a tree-decomposition into finite parts such that the set of induced subgraphs of $G$ is lwqo. Then $G$ has the $\kappa$-EPP for every uncountable cardinal $\kappa$.
\end{restatable}

The conditions on $G$ in Theorem \ref{thm:intro_aleph_0-EPP} are the same as in Theorem \ref{thm:intro_EPP}.
In Theorem \ref{thm:intro_kappa-EPP}, the condition that $G$ is rayless is replaced by the much weaker condition of admitting a tree-decomposition into finite parts.
Indeed, every graph with a normal spanning tree has such a tree-decomposition (see \cite{albrechtsen2024linked}*{Theorem 2.2}) and thus also every graph that does not contain a subdivision of an infinite clique by \cite{halin1978simplicial}*{Theorem 10.1}.
For more graphs with normal spanning trees see \cites{jung1969wurzelbaume, diestel2016simple, pitz2021quickly, pitz2020unified, pitz2021proof}.

Theorems \ref{thm:intro_EPP}, \ref{thm:intro_aleph_0-EPP} and \ref{thm:intro_kappa-EPP} are the main results of this paper. Their proofs rely on the same core ideas and we will prove all three theorems in Section \ref{sec:proof_of_main_thm}.

Similar to Corollary \ref{cor:excluded_path_EPP}, we deduce from Theorems \ref{thm:intro_aleph_0-EPP} and \ref{thm:intro_kappa-EPP}:

\begin{restatable}{cor}{excludedpathkappaEPP}\label{cor:excluded_path_kappa-EPP}
Let $G$ be a graph that does not contain a path of length $n$ as a subgraph for some $n \in \NN$. Then $G$ has the $\kappa$-EPP for every infinite cardinal $\kappa$.
\end{restatable}

\subsection{Graphs without the $\kappa$-EPP for uncountable $\kappa$}

While it is easy to find graphs that do not have the EPP or the $\aleph_0$-EPP (Proposition \ref{prop:ray_does_not_have_EPP}), a bit more work is needed to prove the following theorem:

\begin{restatable}{thm}{counterexamplekappaEPP}\label{thm:counterexample_kappa-EPP}
For every uncountable cardinal $\kappa$ there is a graph that does not have the $\kappa$-EPP.
\end{restatable}

An important step in our proof of this theorem is to construct for each uncountable cardinal $\kappa$ a graph of size $\kappa$ which is not a proper subgraph of itself.
Since it does not make our proof more difficult, we prove the following stronger theorem:

\begin{restatable}{thm}{notproperselfminor}\label{thm:not_proper_selfminor}
For every uncountable cardinal $\kappa$ there is a graph of size $\kappa$ which is not a proper minor of itself.
\end{restatable}

This generalises Oporowski's result \cite{oporowski1990selfminor} that there exists a continuum-sized graph which is not a proper minor of itself.

\subsection{Trees without the $\kappa$-EPP for uncountable $\kappa$}
\label{subsec:trees_without_kappa-EPP}

In an attempt to strengthen Theorem~\ref{thm:counterexample_kappa-EPP}, we ask:

\begin{problem}\label{prob:tree_without_kappa-EPP}
Is there, for every uncountable cardinal $\kappa$, a tree that does not have the $\kappa$-EPP?
\end{problem}

Regarding Problem \ref{prob:tree_without_kappa-EPP}, we are only able to prove a consistency result, which is based on the assumption that there are no weak limit cardinals:

\begin{restatable}{thm}{counterexamplekappaEPPtree}\label{thm:counterexample_kappa-EPP_tree}
It is consistent with ZFC that for every uncountable cardinal $\kappa$ there is a tree that does not have the $\kappa$-EPP.
\end{restatable}

For the proof we need the following theorem, which is based on the same set-theoretic assumption as above:

\begin{restatable}{thm}{notproperselfsubtree}\label{thm:not_proper_self-subtree}
It is consistent with ZFC that for every uncountable cardinal $\kappa$ there is a tree of size $\kappa$ which is not a proper subgraph of itself.
\end{restatable}

It remains open whether this also holds in ZFC:

\begin{problem}\label{prob:not_proper_subtree}
Is there, for every uncountable cardinal $\kappa$, a tree of size $\kappa$ which is not a proper subgraph of itself?
\end{problem}

In Section \ref{sec:trees_without_kappa-EPP} we will see that a positive answer to the following problem would imply a positive answer to the latter:

\begin{problem}\label{prob:subtree_antichain}
Is there, for every uncountable cardinal $\kappa$, a $\kappa$-sized subgraph-antichain of trees of size at most $\kappa$?
\end{problem}

\begin{problem}\label{prob:same_for_rayless_trees}
Do the answers to Problems \ref{prob:tree_without_kappa-EPP}, \ref{prob:not_proper_subtree} or \ref{prob:subtree_antichain} change if we additionally require the trees to be rayless?
\end{problem}

\subsection{The ($\kappa$-)EPP for classes defined by topological minors and minors}

For a graph $G$, write $\cT(G)$ for the class of graphs containing $G$ as a topological minor and $\cM(G)$ for the class of graphs containing $G$ as a minor.
Classes of this form are typically considered in results on the EPP for finite graphs.
In Section \ref{sec:top_minors} (Theorem \ref{thm:main_thm_for_top_minors}), we formulate and prove a version of Theorems \ref{thm:intro_EPP}, \ref{thm:intro_aleph_0-EPP}, and \ref{thm:intro_kappa-EPP} for classes of the form $\cT(G)$.
Note that similar results on the EPP or $\aleph_0$-EPP are not possible for classes of the form $\cM(G)$, since not even $\cM(K_{1, \aleph_0})$ has the EPP or the $\aleph_0$-EPP (Proposition \ref{prop:star_models_do_not_have_EPP}).

Nash-Williams \cite{nashwilliams1965well} showed that the class of all trees is well-quasi-ordered by the topological minor relation and Laver \cite{laver1978better} proved a more general labelled version of this result.
With help of Laver's theorem, we deduce the following result:

\begin{restatable}{cor}{TThasEPP}\label{cor:TT_has_EPP}
$\cT(T)$ has the $\kappa$-EPP for every uncountable cardinal $\kappa$ and every tree $T$. If $T$ is rayless, then $\cT(T)$ also has the $\aleph_0$-EPP.
\end{restatable}

A corresponding statement for the EPP does not hold, since Thomassen \cite{thomassen1988on} showed that there are finite trees $T$ for which $\cT(T)$ does not have the EPP.

An argument similar to Proposition \ref{prop:ray_does_not_have_EPP} shows that $\cT(R)$ and $\cM(R)$, where $R$ is a ray, do not have the EPP or the $\aleph_0$-EPP. However, we know of no such counterexamples for the $\kappa$-EPP for uncountable $\kappa$.

\begin{problem}
Is there an uncountable cardinal $\kappa$ and a graph $G$ such that $\cT(G)$ or $\cM(G)$ does not have the $\kappa$-EPP? Does such a graph exist for every uncountable cardinal $\kappa$?
\end{problem}

\subsection{Comparing and evaluating the results}

Considering Theorems \ref{thm:intro_EPP} and \ref{thm:intro_aleph_0-EPP}, Proposition~\ref{prop:finitely_many_rays}, and Corollary \ref{cor:TT_has_EPP}, the $\aleph_0$-EPP and the EPP seem to behave very similarly.
The $\kappa$-EPP for uncountable $\kappa$ behaves differently.
In fact, it is easy to see that any graph of size less than $\kappa$ has the $\kappa$-EPP (Proposition \ref{small_graphs_have_kappa-EPP}).
So in particular, the ray has the $\kappa$-EPP for all uncountable $\kappa$, but not the EPP or the $\aleph_0$-EPP (Proposition \ref{prop:ray_does_not_have_EPP}).
However, if we restrict ourselves to graphs of size at least $\kappa$, we do not know whether such examples exist:

\begin{problem}
Let $\mu \leq \kappa$ be infinite cardinals and let $G$ be a graph of size at least $\kappa$.
Does $G$ have the $\kappa$-EPP if and only if $G$ has the $\mu$-EPP if and only if $G$ has the EPP?
\end{problem}

Viewed from a distance, the question whether a graph $G$ has the ($\kappa$-)EPP seems to be related to the self-similarity of $G$:
in the proofs of Theorems \ref{thm:intro_EPP}, \ref{thm:intro_aleph_0-EPP}, and \ref{thm:intro_kappa-EPP} the lwqo property is used to find various proper subgraph embeddings of $G$ into itself.
On the other hand, the graphs without the $\kappa$-EPP, which we construct in the proof of Theorem \ref{thm:counterexample_kappa-EPP}, do not admit any proper subgraph embeddings into themselves.

\subsection{Ubiquity}
The results of this paper have applications in the field of ubiquity.
Call a class $\cG$ of graphs ubiquitous if every graph $\Gamma$ that contains $n$ disjoint copies of graphs from $\cG$ as subgraphs for all $n \in \NN$ also contains infinitely many such copies.
It is easy to see that if $\cG$ has the $\aleph_0$-EPP, then $\cG$ is ubiquitous (Proposition \ref{prop:ubiquity}).
However, the converse is not true, since the ray is ubiquitous by a result of Halin \cite{halin1965uber}.

Thus, all graphs as in Theorem \ref{thm:intro_aleph_0-EPP} and Corollary \ref{cor:excluded_path_kappa-EPP} are ubiquitous, which was not known before.
Furthermore, by Corollary \ref{cor:TT_has_EPP} the class $\cT(T)$ is ubiquitous for any rayless tree $T$.
With that, we reobtain a special case of the result of Bowler et.\ al.\ \cite{BEEGHPT22} that $\cT(T)$ is ubiquitous for every tree $T$.

\section{Preliminaries}

All graph-theoretic notations not defined here can be found in Diestel's book \cite{diestel2015book} and all set-theoretic notations in Jech's book \cite{jech2003book}.

\subsection{Minors}
A graph $H$ is a \emph{minor} of a graph $G$ if there is a family $\{ V_h : h \in V(H) \}$ of connected, pairwise disjoint, non-empty subsets of $V(G)$ such that there is a $V_h$--$V_{h'}$ edge in $G$ whenever there is a $h$--$h'$ edge in $H$. The sets $V_h$ for $h \in V(H)$ are called \emph{branch sets}.

\subsection{Tree order}
All trees in this paper have a fixed root vertex, which we will usually not explicitly specify.
(In some parts of the paper, root vertices of trees are not needed and can be ignored.)
Given a tree $T$ with root $r$, we define a partial order on $V(T)$ by $x \leq y$ if $x$ lies on the unique $r$--$y$ path in $T$.
We denote by $\br_T(x)$ the subtree of $T$ with vertex set $\{ y \in V(T) : x \leq y \}$.
For $x \in V(T)$, we write $\successor_T(x)$ for the set of immediate successors of $x$ in $V(T)$.
Furthermore, we call a subtree of $T$ \emph{rooted} if it contains $r$.

\subsection{Tree-decompositions}

A \emph{tree-decomposition} of a graph $G$ is a pair $(T, \cV)$, where $T$ is a tree and $\cV = (V_t : t \in V(T))$ is a family of subsets of $V(G)$ such that:

\begin{itemize}
\item $V(G) = \bigcup \{ V_t : t \in V(T) \}$,
\item for every edge $vw$ of $G$ there is a $t \in V(T)$ such that $v \in V_t$ and $w \in V_t$, and
\item if $t_1, t_2, t_3 \in V(T)$ such that $t_2$ lies on the unique $t_1$--$t_3$ path in $T$, then $V_{t_1} \cap V_{t_3} \subseteq V_{t_2}$.
\end{itemize}
The sets $V_t$ for $t \in V(T)$ are called \emph{parts} of the tree-decomposition.

Halin \cite{halin1966graphen} showed:

\begin{lemma}\label{lem:rayless_tree-decomposition}
Every rayless graph admits a tree-decomposition $(T, \cV)$ into finite parts such that $T$ is rayless.
\end{lemma}

\subsection{Compactness}\label{subsec:compactness}

Following Diestel's book \cite{diestel2015book}*{Appendix A}, we describe the Compactness Principle, a combinatorial framework for compactness proofs.

Let $V$ be an arbitrary set, let $\cF$ be a set of finite subsets of $V$, and let $S$ be any finite set.
Suppose for each $U \in \cF$ we have fixed a set of \emph{admissible} functions $U \to S$.
We call a subset $\cU$ of $\cF$ \emph{compatible} if there is a function $g: V \to \cF$ such that $g \upharpoonright U$ is admissible for all $U \in \cU$.

\begin{thm}[Compactness Principle]\label{thm:compactness_principle}
If all finite subsets of $\cF$ are compatible, then $\cF$ is compatible.
\end{thm}

\section{Propositions}\label{sec:propositions}

We say that a class $\cG$ of finite graphs has the \emph{EPP for finite host graphs} if it has the EPP as defined in the beginning of Section \ref{sec:introduction} with the additional requirement that the host graph $\Gamma$ is finite.

\begin{prop}\label{prop:finite_EPP}
Any class $\cG$ of finite graphs has the EPP for finite host graphs if and only if it has the EPP.
\end{prop}

\begin{proof}
If $\cG$ has the EPP, then it clearly has the EPP for finite host graphs.
Conversely, let $f: \NN \to \NN$ witness that $\cG$ has the EPP for finite host graphs.
We show that the same function witnesses that $\cG$ has the EPP.
Let $k \in \NN$ and let $\Gamma$ be any infinite graph that does not contains $k$ disjoint copies of graphs from $\cG$.
We have to find a set $X \subseteq V(\Gamma)$ of size at most $f(k)$ such that $\Gamma - X$ does not contain any copy of a graph from $\cG$.
This is achieved by a standard compactness argument, for which we use the Compactness Principle (see Section \ref{subsec:compactness}). We provide the details in the following.

Let $\cF$ be the set of all finite subsets of $V(\Gamma)$ and let $U \in \cF$.
We call a function $g: U \to \{0, 1\}$ admissible if $X' := g^{-1}\{1\}$ has size at most $f(k)$ and $\Gamma[U] - X'$ does not contain any copy of a graph from $\cG$.
We show that every finite $\cU \subseteq \cF$ is compatible.
Indeed, since $U^* := \bigcup \cU \subseteq V(\Gamma)$ is finite and since $\cG$ has the EPP for finite host graphs, there is a set $X^* \subseteq U^*$ of size at most $f(k)$ such that $\Gamma[U^*] - X^*$ does not contain any copy of a graph from $\cG$.
Then the characteristic function $g^* : V(\Gamma) \to \{ 0, 1 \}$ of $X^*$ witnesses that $\cU$ is compatible.
By the Compactness Principle (Theorem \ref{thm:compactness_principle}), also $\cF$ is compatible; let $h: V(\Gamma) \to \{ 0, 1 \}$ witness compatibility of $\cF$.
It is straightforward to check that $X := h^{-1}\{1\}$ has size at most $f(k)$ and that $\Gamma - X$ does not contain any copy of a member of $\cG$.
\end{proof}

\begin{prop}\label{prop:finite_graphs_have_EPP}
Every finite graph $G$ has the EPP.
\end{prop}

\begin{proof}
Let $f : \NN \to \NN, k \mapsto |G| \cdot (k - 1)$.
Let $k \in \NN$ and consider any graph $\Gamma$ such that $\Gamma - X$ contains a copy of $G$ for all $X \subseteq V(\Gamma)$ of size at most $f(k)$.
We recursively find $k$ disjoint copies of $G$ in $\Gamma$.
Having found disjoint copies $G_1, \dots, G_i$ for $i < k$, we find another disjoint copy by choosing $X := V(G_1) \cup \dots \cup V(G_i)$, which is a set of size at most $f(k)$.
\end{proof}

\begin{prop}\label{small_graphs_have_kappa-EPP}
For every infinite cardinal $\kappa$, every graph $G$ of size less than $\kappa$ has the EPP.
\end{prop}

\begin{proof}
Let $\Gamma$ be any graph such that $\Gamma - X$ contains a copy of $G$ for all $X \subseteq V(\Gamma)$ of size less than $\kappa$.
We recursively construct a family $(G_\alpha)_{\alpha < \kappa}$ of pairwise disjoint copies of $G$ in $\Gamma$.
If $\alpha < \kappa$ and we have defined $G_\beta$ for all $\beta < \alpha$, then we find $G_\alpha$ by choosing $X := \bigcup \{ V(G_\beta) : \beta < \alpha \}$, which is a set of size less than $\kappa$.
\end{proof}

\begin{prop}\label{prop:ray_does_not_have_EPP}
The ray does not have the EPP or the $\aleph_0$-EPP.
\end{prop}

\begin{proof}
Consider a ray as host graph $\Gamma$.
Then $\Gamma$ does not contain contain 2 disjoint rays.
However, for every finite $X \subseteq V(\Gamma)$ the graph $\Gamma - X$ still contains a ray.
Therefore, the ray does not have the EPP or the $\aleph_0$-EPP (the former holds because it is not possible to define $f(2)$).
\end{proof}

For Proposition \ref{prop:finitely_many_rays}, we need the following result by Halin \cite{halin1975some}*{Theorem 2}:

\begin{lemma}\label{lem:ray-decomposition}
Let $G$ be a connected infinite locally finite graph such that the maximum number of disjoint rays in $G$ is $n \in \NN$.
Then there is a tree-decomposition $(R, \cV)$ of $G$ into finite parts where $R$ is a ray such that all adhesion sets have size $n$ and are pairwise disjoint.
\end{lemma}

\begin{prop}\label{prop:finitely_many_rays}
Let $G$ be a connected infinite locally finite graph that does not contain infinitely many disjoint rays. Then $G$ does not have the EPP or the $\aleph_0$-EPP.
\end{prop}

\begin{proof}
Halin showed in \cite{halin1965uber} that the ray is ubiquitous.
Therefore, since $G$ does not contain infinitely many disjoint rays, there is a maximum number $n \in \NN$ of disjoint rays in $G$.

Consider the tree-decomposition $(R, \cV)$ of $G$ from Lemma \ref{lem:ray-decomposition} where $R = r_0 r_1 r_2 \dots$.
Let $G_0, G_1, G_2, \dots$ be infinitely many copies of $G$ such that:

\begin{itemize}
\item For all $j < k \in \NN$ the sets of vertices of $G_j$ and $G_k$ corresponding to $V(G) \setminus \bigcup \{ V_{r_i} : i \geq k \}$ are disjoint.
\item For all $j < k \leq i \in \NN$ and all $v \in V_{r_i}$, the vertex of $G_j$ corresponding to $v$ coincides with the vertex of $G_k$ corresponding to $v$.
\end{itemize}

Let $\Gamma := \bigcup_{k \in \NN} G_k$ (see Figure \ref{fig:finitely_many_rays}).
Since $\bigcap_{k \in \NN} G_k = \emptyset$, for every finite $X \subseteq V(\Gamma)$ there is $k \in \NN$ such that $G_k$ avoids $X$.

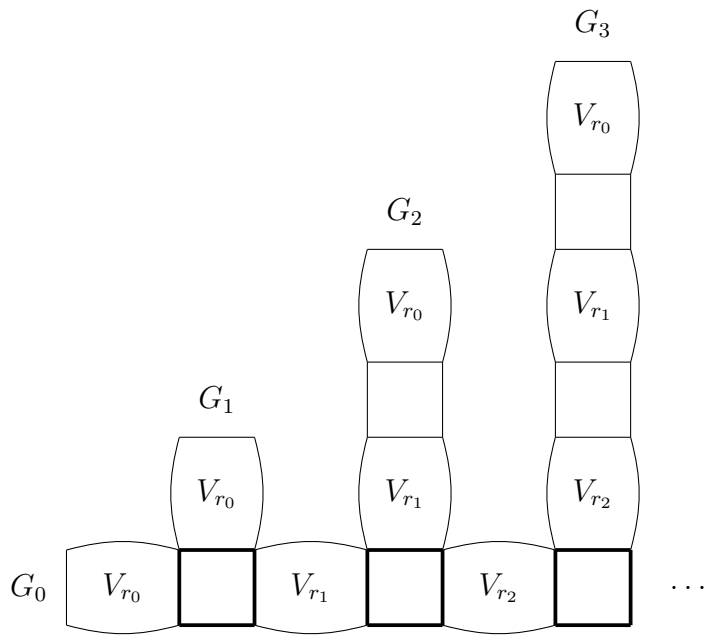
\begin{figure}
\center
\begin{tikzpicture}

\draw (-1.5,0) to (-1.5,1);
\draw[bend right=15] (-1.5,0) to (0,0);
\draw[bend left=15] (-1.5,1) to (0,1);

\draw[line width=0.5mm] (0,0) to (0,1);
\draw[line width=0.5mm] (0,1) to (1,1);
\draw[line width=0.5mm] (1,1) to (1,0);
\draw[line width=0.5mm] (1,0) to (0,0);

\draw[line width=0.5mm] (2.5,0) to (2.5,1);
\draw[line width=0.5mm] (2.5,1) to (3.5,1);
\draw[line width=0.5mm] (3.5,1) to (3.5,0);
\draw[line width=0.5mm] (3.5,0) to (2.5,0);

\draw[bend right=15] (1,0) to (2.5,0);
\draw[bend left=15] (1,1) to (2.5,1);

\draw[line width=0.5mm] (5,0) to (5,1);
\draw[line width=0.5mm] (5,1) to (6,1);
\draw[line width=0.5mm] (6,1) to (6,0);
\draw[line width=0.5mm] (6,0) to (5,0);

\draw[bend right=15] (3.5,0) to (5,0);
\draw[bend left=15] (3.5,1) to (5,1);

\draw (0,1)[bend left=15] to (0,2.5);
\draw (1,1)[bend right=15] to (1,2.5);
\draw (0,2.5) to (1,2.5);

\draw (2.5,1)[bend left=15] to (2.5,2.5);
\draw (3.5,1)[bend right=15] to (3.5,2.5);
\draw (2.5,2.5) to (3.5,2.5);

\draw (2.5,2.5) to (2.5,3.5);
\draw (3.5,2.5) to (3.5,3.5);
\draw (2.5,3.5) to (3.5,3.5);

\draw (2.5,3.5)[bend left=15] to (2.5,5);
\draw (3.5,3.5)[bend right=15] to (3.5,5);
\draw (2.5,5) to (3.5,5);

\draw (5,1)[bend left=15] to (5,2.5);
\draw (6,1)[bend right=15] to (6,2.5);
\draw (5,2.5) to (6,2.5);

\draw (5,2.5) to (5,3.5);
\draw (6,2.5) to (6,3.5);
\draw (5,3.5) to (6,3.5);

\draw (5,3.5)[bend left=15] to (5,5);
\draw (6,3.5)[bend right=15] to (6,5);
\draw (5,5) to (6,5);

\draw (5,5) to (5,6);
\draw (6,5) to (6,6);
\draw (5,6) to (6,6);

\draw (5,6)[bend left=15] to (5,7.5);
\draw (6,6)[bend right=15] to (6,7.5);
\draw (5,7.5) to (6,7.5);

\node at (6.8,0.5){$\dots$};

\node at (-2,0.5){$G_0$};
\node at (0.5,3){$G_1$};
\node at (3,5.5){$G_2$};
\node at (5.5,8){$G_3$};

\node at (-0.75,0.5){$V_{r_0}$};
\node at (1.75,0.5){$V_{r_1}$};
\node at (4.25,0.5){$V_{r_2}$};

\node at (0.5,1.75){$V_{r_0}$};

\node at (3,1.75){$V_{r_1}$};
\node at (3,4.25){$V_{r_0}$};

\node at (5.5,1.75){$V_{r_2}$};
\node at (5.5,4.25){$V_{r_1}$};
\node at (5.5,6.75){$V_{r_0}$};

\end{tikzpicture}
\caption{The graph $\Gamma$ from the proof of Proposition \ref{prop:finitely_many_rays}. Adhesion sets in the graphs $G_i$ are depicted by squares and each part $V_{r_i}$ consists of the respective labelled area together with its adjacent squares. Note that every ray in $\Gamma$ meets all but finitely many of the thickly drawn adhesion sets.}
\label{fig:finitely_many_rays}
\end{figure}

Furthermore, it is clear from Figure \ref{fig:finitely_many_rays} that for every ray $S$ in $\Gamma$ there is $i \in \NN$ such that for all $j \geq i$, the ray $S$ meets the set of vertices of $G_0$ corresponding to the adhesion set $V_{r_j} \cap V_{r_{j + 1}}$.
As all adhesion sets have size $n$, the graph $\Gamma$ cannot contain more than $n$ disjoint rays.
Therefore, $\Gamma$ does not contain two disjoint copies of $G$.
Since we have also seen that $\Gamma - X$ contains a copy of $G$ for all finite $X \subseteq V(\Gamma)$,
it follows that $G$ does not have the EPP or the $\aleph_0$-EPP.
\end{proof}

\begin{prop}\label{prop:star_models_do_not_have_EPP}
$\cM(K_{1, \aleph_0})$ does not have the EPP or the $\aleph_0$-EPP.
\end{prop}

\begin{proof}
Let $\Gamma$ be the infinite comb (Figure \ref{fig:inf_comb}).
Then $\Gamma$ does not contain 2 disjoint copies of $K_{1, \aleph_0}$ as minors, but for every finite set $X \subseteq V(\Gamma)$, the graph $\Gamma - X$ contains $K_{1, \aleph_0}$ as a minor.
Thus $\cM(K_{1, \aleph_0})$ does not have the EPP or the $\aleph_0$-EPP.
\end{proof}

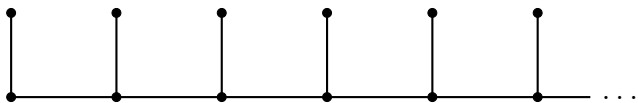
\begin{figure}[ht]
\center
\begin{tikzpicture}[scale=1.4]
	\foreach \x in {0,1,2,3,4,5} \foreach \y in {0,0.8} \draw[fill] (\x,\y) circle (1.2pt);
	\foreach \x in {0,1,2,3,4,5} \draw[thick] (\x,0) -- (\x,0.8);
	\draw[thick] (0,0) -- (5.5,0) node[right] {$\dots$};
\end{tikzpicture}
\caption{The infinite comb.}
\label{fig:inf_comb}
\end{figure}

\begin{prop}\label{prop:ubiquity}
If a class of graphs has the $\aleph_0$-EPP, then it is ubiquitous.
\end{prop}

\begin{proof}
Let $\cG$ be any class of graphs that has the $\aleph_0$-EPP.
To see that $\cG$ is ubiquitous, let $\Gamma$ be any graph that contains $n$ disjoint copies of graphs from $\cG$ for all $n \in \NN$.
Then $\Gamma - X$ contains a copy of $G$ for every finite set $X \subseteq V(G)$, and it follows from the $\aleph_0$-EPP that $\Gamma$ contains infinitely many disjoint copies of graphs from $\cG$.
\end{proof}

\section{Constructing counterexamples to Seymour's Self-Minor Conjecture}\label{sec:self-minors}

A graph $H$ is a \emph{proper minor} of a graph $G$ if $H$ is a minor of some proper subgraph of $G$. In this section, we prove:

\notproperselfminor*

Our construction relies on the existence of large minor-antichains of infinite graphs. Their existence is proved by Komjáth \cite{komjath1994minors} (also see \cite{pitz2023minor}):

\begin{thm}\label{thm:minor_antichain}
For every uncountable cardinal $\kappa$, there are $2^\kappa$ graphs of size $\kappa$ none of them being a minor of another.
\end{thm}

Let us briefly look at the context of Theorem \ref{thm:not_proper_selfminor}.
Seymour's Self-Minor Conjecture states that every infinite graph is a proper minor of itself.
While it is still open whether the conjecture holds for countable graphs, Oporowski \cite{oporowski1990selfminor} disproved Seymour's conjecture by constructing a continuum-sized graph which is not a proper minor of itself.
In Theorem \ref{thm:not_proper_selfminor} we show more generally that there exist counterexamples of all uncountable cardinalities.
It should be noted that Oporowski's result (1990) is older than Komjáth's Theorem \ref{thm:minor_antichain} (1995), and instead relies on a weaker result by Thomas \cite{thomas1988counterexample} from 1988.
Our construction is different from Oporowski's and slightly simpler, but some of the ideas and tools we use are the same:

\begin{lemma}[\cite{oporowski1990selfminor}, Lemma 1]\label{lem:adding_dominating_vertex}
Let $G'$ be obtained from $G$ by adding a new vertex adjacent to every vertex of $G$, and let $H'$ be obtained from $H$ similarly. Then $H$ is a minor of $G$ if and only if $H'$ is a minor of $G'$.
\end{lemma}

A \emph{block} of a graph is a maximal connected subgraph that has no cutvertex.

\begin{lemma}[\cite{oporowski1990selfminor}, Lemma 3]\label{lem:2-connected_minor}
Let $H$ be a minor of $G$ with branch sets $\{V_h : h \in V(H)\}$. Then for every block $A$ of $H$ there is a block $B$ of $G$ such that $A$ is a minor of $B$ with branch sets $\{V_h \cap V(B) : h \in V(A)\}$.
\end{lemma}

\begin{proof}[Proof of Theorem \ref{thm:not_proper_selfminor}]
By Theorem \ref{thm:minor_antichain} there is a $\kappa$-sized minor-antichain of graphs of size $\kappa$.
For every graph in this antichain, add two vertices which are adjacent to all vertices of this graph and to each other.
The resulting set $\cH$ of graphs is again a minor-antichain by Lemma \ref{lem:adding_dominating_vertex} and all graphs in $\cH$ are 2-connected.
Suppose without loss of generality that the graphs in $\cH$ are pairwise disjoint.

Let $T$ be a $\kappa$-regular tree and fix a bijection $f: V(T) \to \cH$.
Moreover, fix for every $t \in V(T)$ a bijection $g_t : N_T(t) \to V(f(t))$, which is possible since $N_T(t)$ and $V(f(t))$ are both sets of size $\kappa$.
We construct a graph $G$ from $\bigcup \cH$ by identifying for every edge $tu \in E(T)$ the vertices $g_t(u)$ and $g_u(t)$.
Note that the elements of $\cH$ are, up to isomorphism, precisely the blocks of $G$, and every $x \in V(G)$ is contained in precisely two blocks of $G$.

Now suppose that $G$ is a minor of itself with branch sets $\{ V_x : x \in V(G) \}$.
To show that $G$ is not a proper minor of itself, we prove the stronger assertion that $V_x = \{ x \}$ for all $x \in V(G)$.

By Lemma \ref{lem:2-connected_minor}, for every block $A$ of $G$ there is a block $B$ of $G$ such that $A$ is a minor of $B$ with branch sets $\{ V_x \cap V(B) : x \in V(A) \}$.
Since $\cH$ is a minor-antichain, we must have $A = B$.
In particular, $V_x \cap V(A) \neq \emptyset$ for all $x \in V(A)$.

Consider any vertex $x \in V(G)$ and let $A, A'$ be the two blocks of $G$ containing $x$.
As observed above, $V_x$ meets both $A$ and $A'$.
Since $V_x$ is connected and $x$ separates $A$ from $A'$ in $G$, we must have $x \in V_x$.
Since the same holds for all $x \in V(G)$ and since distinct branch sets are disjoint, it follows that $V_x = \{ x \}$.
\end{proof}

\section{Graphs without the $\kappa$-EPP}\label{sec:graphs_without_kappa-EPP}

In this section, we prove:

\counterexamplekappaEPP*

We will obtain such a graph by applying the following operation to any $\kappa$-sized graph which is not a proper subgraph of itself.
For a graph $G$, we write $G^\mathsf{V}$ for the graph obtained from $G$ by adding two new vertices $v_1, v_2$ for each $v \in V(G)$ and adding edges $vv_1$ and $vv_2$. Note that every vertex of $G^\mathsf{V}$ has either degree 1 or degree at least 3, and the set of vertices of degree at least 3 of $G^\mathsf{V}$ is $V(G)$.

The following lemma is straightforward and we omit the proof:
\begin{lemma}\label{lem:G^V_is_subgraph}
A graph $H$ is a subgraph of a graph $G$ if and only if $H^\mathsf{V}$ is a subgraph of $G^\mathsf{V}$.
\end{lemma}

\begin{thm}\label{thm:graph_without_kappa-EPP}
Let $\kappa$ be an infinite cardinal and let $G$ be a graph of size at least $\kappa$ which is not a proper subgraph of itself. Then $G^\mathsf{V}$ does not have the $\kappa$-EPP.
\end{thm}

\begin{proof}
Consider a set $\cS$ consisting of $\kappa$ copies of $G^\mathsf{V}$ such that:
\begin{enumerate}[label=(\roman*)]
\item for all $H \neq H' \in \cS$, the set of vertices of degree at least 3 of $H$ is disjoint from the set of vertices of degree at least 3 of $H'$,
\item for all $H \neq H' \in \cS$, a degree 1 vertex of $H$ coincides with a degree 1 vertex of $H'$, and
\item every vertex of $\bigcup \cS$ is contained in at most two distinct graphs from $\cS$.
\end{enumerate}
Since $G^\mathsf{V}$ has $\kappa$ vertices of degree 1, such a set $\cS$ can easily be obtained by beginning with $\kappa$ disjoint copies of $G^\mathsf{V}$ and recursively identifying pairs of degree 1 vertices.

We show that $\Gamma := \bigcup \cS$ witnesses that $G^\mathsf{V}$ does not have the $\kappa$-EPP.
Indeed, by (iii), $\Gamma - X$ contains a copy of $G^\mathsf{V}$ for all $X \subseteq V(\Gamma)$ of size less than $\kappa$.
It is left to show that $\Gamma$ does not contain two disjoint copies of $G^\mathsf{V}$.
Note that by (i) and (iii), every vertex of $\Gamma$ that is contained in more than one graph from $\cS$ has degree 2.
Therefore, since $G^\mathsf{V}$ has no degree 2 vertices, any copy of $G^\mathsf{V}$ in $\Gamma$ is completely contained in some member of $\cS$.
As $G^\mathsf{V}$ is not a proper subgraph of itself by Lemma \ref{lem:G^V_is_subgraph}, any copy of $G^\mathsf{V}$ in $\Gamma$ must even coincides with a member of $\cS$.
Thus $\Gamma$ does not contain two disjoint copies of $G^\mathsf{V}$ by (ii).
\end{proof}

\begin{proof}[Proof of Theorem \ref{thm:counterexample_kappa-EPP}]
By Theorem \ref{thm:not_proper_selfminor}, there is a graph $G$ of size $\kappa$ which is not a proper minor of itself and in particular not a proper subgraph of itself.
Then $G^\mathsf{V}$ does not have the $\kappa$-EPP by Theorem \ref{thm:graph_without_kappa-EPP}.
\end{proof}

\section{Trees without the $\kappa$-EPP}\label{sec:trees_without_kappa-EPP}

Using a similar strategy as in the proof of Theorem \ref{thm:counterexample_kappa-EPP} given in the previous two sections, we show:

\counterexamplekappaEPPtree*

The set-theoretic assumption we need is that there exist no regular limit cardinals.
This assumption is used in Theorem \ref{thm:antichain_of_trees}, in which we construct a $2^\kappa$-sized subgraph-antichain of $\kappa$-sized trees.
Using the graphs from this antichain as building blocks, we construct a tree $T$ of size $\kappa$ which is not a proper subgraph of itself in Theorem \ref{thm:not_proper_self-subtree}.
Finally, we deduce that $T^\mathsf{V}$ does not have the $\kappa$-EPP.
We do not know if any of these theorems also hold in ZFC (see Problems \ref{prob:tree_without_kappa-EPP}, \ref{prob:not_proper_subtree}, \ref{prob:subtree_antichain} and \ref{prob:same_for_rayless_trees}).

We begin by proving two lemmas that are needed for Theorem \ref{thm:antichain_of_trees}. Our proof of Theorem \ref{thm:antichain_of_trees}, including the two lemmas, is based on ideas from Komjáth's proof of Theorem \ref{thm:minor_antichain} \cite{komjath1994minors}.

\begin{lemma}\label{lem:antichains_of_sets}
Let $S$ be any set of infinite cardinality $\kappa$. Then there is a $2^\kappa$-sized set $\cP$ of $\kappa$-sized subsets of $S$ such that no element of $\cP$ is a subset of another.
\end{lemma}

\begin{proof}
Let $P$ be a partition of $S$ into 2-element subsets and note that $|P| = |S|$ since $S$ is infinite.
Then the set $\cP$ of all $\kappa$-sized subsets of $S$ that contain exactly one of the two elements of all sets in $P$ is as required.
\end{proof}

\begin{lemma}\label{lem:kappa->2^kappa}
Let $\kappa$ be any infinite cardinal. If there exists a $\kappa$-sized subgraph-antichain of (rayless) trees of size at most $\kappa$, then there exists a $2^\kappa$-sized subgraph-antichain of (rayless) trees of size $\kappa$.
\end{lemma}

\begin{proof}
Let $\cU$ be a $\kappa$-sized subgraph-antichain of pairwise disjoint trees of size at most $\kappa$ and let $\cS$ be a $2^\kappa$-sized $\subseteq$-antichain of $\kappa$-sized subsets of $\cU$, which exists by Lemma \ref{lem:antichains_of_sets}.
For every $S \in \cS$, let $T_S$ be a tree obtained from $\bigcup S$ by adding a new vertex $v_S$ and joining it to one vertex of each tree in $S$.
We claim that $\cT := \{T_S : S \in \cS \}$ is the desired subgraph-antichain.

Clearly, the trees in $\cT$ have size $\kappa$, and if the trees in $\cU$ are rayless, then the trees in $\cT$ are rayless as well. Finally, let $S \neq S' \in \cS$ and suppose for a contradiction that there is a subgraph embedding $f: T_S \to T_{S'}$.
We cannot have $f(v_S) = v_{S'}$ since then $f$ would have to send the trees from $S$ injectively to trees from $S'$, but $S \not\subseteq S'$ and $\cU$ is a subgraph-antichain.
Hence $f(v_S) \in U$ for some tree $U \in S'$.
But since the vertex $v_{S'}$ separates $U$ from all other elements of $S'$ in $T_{S'}$, it follows that $f$ embeds all but at most one tree from $S$ into $U$, contradicting that $\cU$ is a subgraph-antichain.
\end{proof}

\begin{thm}\label{thm:antichain_of_trees}
It is consistent with ZFC that for every infinite cardinal $\kappa$, there are $2^\kappa$ rayless trees of size $\kappa$ none of them being a subgraph of another.
\end{thm}

\begin{proof}
We assume that there exist no regular limit cardinals, which are also known as weakly inaccessible cardinals and whose existence is not provable in ZFC (see \cite{jech2003book}).
Under this assumption, we prove the theorem by induction on $\kappa$.

To find a $2^{\aleph_0}$-sized antichain of $\aleph_0$-sized rayless trees, apply Lemma \ref{lem:kappa->2^kappa} to the set of double-ended forks shown in Figure \ref{fig:double-ended_forks}.

\begin{figure}[ht]
\begin{tikzpicture}[scale=0.6]

\tikzstyle{v}=[draw,circle,color=black,fill=black,minimum size=2.7pt,inner sep=0pt]

\node[v] (a0) at (-1,-1) {};
\node[v] (a1) at (-1,1) {};
\node[v] (a2) at (0,0) {};
\node[v] (a3) at (1,-1) {};
\node[v] (a4) at (1,1) {};

\draw (a0) to (a2) {};
\draw (a1) to (a2) {};
\draw (a2) to (a3) {};
\draw (a2) to (a4) {};

\node[v] (b0) at (3,-1) {};
\node[v] (b1) at (3,1) {};
\node[v] (b2) at (4,0) {};
\node[v] (b3) at (5,0) {};
\node[v] (b4) at (6,-1) {};
\node[v] (b5) at (6,1) {};

\draw (b0) to (b2) {};
\draw (b1) to (b2) {};
\draw (b2) to (b3) {};
\draw (b3) to (b4) {};
\draw (b3) to (b5) {};

\node[v] (c0) at (8,-1) {};
\node[v] (c1) at (8,1) {};
\node[v] (c2) at (9,0) {};
\node[v] (c3) at (10,0) {};
\node[v] (c4) at (11,0) {};
\node[v] (c5) at (12,-1) {};
\node[v] (c6) at (12,1) {};

\draw (c0) to (c2) {};
\draw (c1) to (c2) {};
\draw (c2) to (c3) {};
\draw (c3) to (c4) {};
\draw (c4) to (c5) {};
\draw (c4) to (c6) {};

\node[v] (d0) at (14,-1) {};
\node[v] (d1) at (14,1) {};
\node[v] (d2) at (15,0) {};
\node[v] (d3) at (16,0) {};
\node[v] (d4) at (17,0) {};
\node[v] (d5) at (18,0) {};
\node[v] (d6) at (19,-1) {};
\node[v] (d7) at (19,1) {};

\draw (d0) to (d2) {};
\draw (d1) to (d2) {};
\draw (d2) to (d3) {};
\draw (d3) to (d4) {};
\draw (d4) to (d5) {};
\draw (d5) to (d6) {};
\draw (d5) to (d7) {};

\node at (21,0) {$\dots$};

\end{tikzpicture}
\caption{A countable subgraph-antichain of finite trees.}
\label{fig:double-ended_forks}
\end{figure}
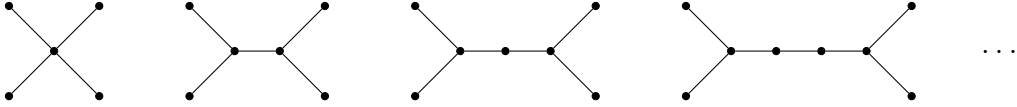

Now let $\kappa > \aleph_0$ and suppose that a $2^\kappa$-sized antichain $\cT$ of $\kappa$-sized rayless trees is given.
Since $|\cT| = 2^\kappa \geq \kappa^+$, there exists a $2^{(\kappa^+)}$-sized antichain of $\kappa^+$-sized rayless trees by Lemma \ref{lem:kappa->2^kappa}.

Finally, suppose that $\kappa$ is a limit cardinal and that the theorem holds for all cardinals less than $\kappa$.
By the assumption that there are no regular limit cardinals, $\kappa$ is singular.
Consider an increasing cofinal sequence of cardinals $(\kappa_\alpha : \alpha < \cf(\kappa))$ in $\kappa$ with $\cf(\kappa) < \kappa_0$.
Let $\{ T_\alpha : \alpha < \cf(\kappa) \}$ be a subgraph-antichain of $\cf(\kappa)$-sized rayless trees, which exists by the induction hypothesis.
Additionally, consider for every $\alpha < \cf(\kappa)$ a subgraph-antichain $\{ T_{\alpha, \beta} : \beta < \kappa_\alpha \}$ of $\kappa_\alpha$-sized rayless trees, which exists again by the induction hypothesis.

For all $\alpha < \cf(\kappa)$ and $\beta < \kappa_\alpha$, let $T^*_{\alpha, \beta}$ be a tree obtained from the disjoint union of a $\kappa$-star with root $r_{\alpha, \beta}$ and $T_{\alpha, \beta}$ and $T_\alpha$ by adding an edge between $r_{\alpha, \beta}$ and an arbitrary vertex of $T_{\alpha, \beta}$ and an edge between $r_{\alpha, \beta}$ and an arbitrary vertex of $T_\alpha$ (see Figure \ref{fig:T*}).
We show that $\{ T^*_{\alpha, \beta} : \alpha < \cf(\kappa), \beta < \kappa_\alpha \}$ is a subgraph-antichain.
Since this antichain has size $\kappa$ and consists of rayless trees of size $\kappa$, the assertion of the theorem then follows from Lemma \ref{lem:kappa->2^kappa}.

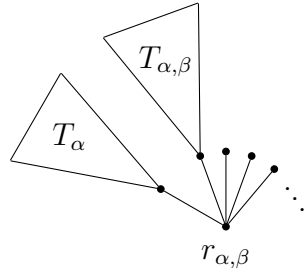
\begin{figure}[ht]
\center
\begin{tikzpicture}[scale=1]

\tikzstyle{v}=[draw,circle,color=black,fill=black,minimum size=2.5pt,inner sep=0pt]
\tikzstyle{u}=[draw,circle,color=black,fill=black,minimum size=0pt,inner sep=0pt]

\node[v] (r) at (0,0) {};
\node[v] (v1) at (150:1) {};
\node[v] (v2) at (110:1) {};
\node[v] (v3) at (90:1) {};
\node[v] (v4) at (70:1) {};
\node[v] (v5) at (50:1) {};

\node[u] (u1) at (163:3) {};
\node[u] (u2) at (137:3) {};
\node[u] (u3) at (123:3) {};
\node[u] (u4) at (97:3) {};

\draw (r) to (v1) {};
\draw (r) to (v2) {};
\draw (r) to (v3) {};
\draw (r) to (v4) {};
\draw (r) to (v5) {};

\draw (v1) to (u1) {};
\draw (v1) to (u2) {};
\draw (u1) to (u2) {};
\draw (v2) to (u3) {};
\draw (v2) to (u4) {};
\draw (u3) to (u4) {};

\node[rotate = 125] at (0.9,0.4) {\small{$\dots$}};
\node at (-2.08,1.2) {$T_\alpha$};
\node at (-0.8,2.2) {$T_{\alpha, \beta}$};
\node at (0,-0.4) {$r_{\alpha, \beta}$};

\end{tikzpicture}
\caption{The tree $T^*_{\alpha, \beta}$ from the proof of Theorem \ref{thm:antichain_of_trees}.}
\label{fig:T*}
\end{figure}

Indeed, consider any $\alpha, \alpha' < \cf(\kappa)$ and $\beta < \kappa_\alpha$ and $\beta' < \kappa_{\alpha'}$ with $(\alpha, \beta) \neq (\alpha', \beta')$ and suppose for a contradiction that there is a subgraph embedding $f : T^*_{\alpha, \beta} \to T^*_{\alpha', \beta'}$.
Note that $f$ sends $r_{\alpha, \beta}$ to $r_{\alpha', \beta'}$, since these are the unique vertices of degree $\kappa$.
Since $T_{\alpha, \beta}$ and $T_{\alpha', \beta'}$ are the only components of $T^*_{\alpha, \beta} - r_{\alpha, \beta}$ and $T^*_{\alpha', \beta'} - r_{\alpha', \beta'}$ of size greater than $\cf(\kappa)$, the map $f$ sends $T_{\alpha, \beta}$ into $T_{\alpha', \beta'}$, and consequently also $T_\alpha$ into $T_{\alpha'}$.
However, if $\alpha \neq \alpha'$, then $T_\alpha$ is not a subgraph of $T_{\alpha'}$, and if $\alpha = \alpha'$ but $\beta \neq \beta'$, then $T_{\alpha, \beta}$ is not a subgraph of $T_{\alpha', \beta'}$, a contradiction.
\end{proof}

Using a similar construction as in the proof of Theorem \ref{thm:not_proper_selfminor}, we prove:

\notproperselfsubtree*

\begin{proof}
We assume that there exists a $\kappa$-sized subgraph-antichain $\cU'$ of $\kappa$-sized trees, which is consistent with ZFC by Theorem \ref{thm:antichain_of_trees}.
Then also $\cU := \{ U^\mathsf{V} : U \in \cU' \}$ is an antichain by Lemma~\ref{lem:G^V_is_subgraph}.
Without loss of generality the elements of $\cU$ are pairwise disjoint.
Note that
\begin{enumerate}[label=(\arabic*)]
\item\label{item:no-adjacent_degree_leq_2_vertices} no tree in $\cU$ contains two adjacent vertices of degree at most 2.
\end{enumerate}

Let $T$ be a $\kappa$-regular tree and fix a bijection $f: V(T) \to \cU$.
Moreover, fix for every $t \in V(T)$ a bijection $g_t : N_T(t) \to L(f(t))$, where $L(f(t))$ denotes the ($\kappa$-sized) set of leaves of the tree $f(t)$.
We construct a tree $U^*$ from $\bigcup \cU$ by adding a $g_t(u)$--$g_u(t)$ edge for every edge $tu \in E(T)$.
Thus for every $U \in \cU$ and every leaf $\l$ of $U$, there is exactly one edge in the tree $U^*$ connecting $\l$ to some leaf of another tree from $\cU$.

We call the edges of $E(U^*) \setminus E(\bigcup \cU)$ \emph{new edges} of $U^*$.
Since new edges only connect leaves of trees from $\cU$,
\begin{enumerate}[resume, label=(\arabic*)]
\item\label{item:endvertices_have_degree_2} all endvertices of new edges have degree 2 in $U^*$.
\end{enumerate}
Furthermore, since every vertex of $U^*$ is either a leaf of some $U \in \cU$ or separates two leaves of some $U \in \cU$,
\begin{enumerate}[resume, label=(\arabic*)]
\item\label{item:all_vertices_separate} every vertex of $U^*$ separates two trees from $\cU$ in $U^*$.
\end{enumerate}

Let $h$ be any subgraph embedding of $U^*$ into $U^*$ and consider any $U \in \cU$.
By \ref{item:no-adjacent_degree_leq_2_vertices} and \ref{item:endvertices_have_degree_2}, the tree $h(U)$ does not contain any new edges of $U^*$.
Hence there is $U' \in \cU$ such that $h(U)$ is completely contained in $U'$.
Since $\cU$ is an antichain, we have $U' = U$.
In particular, the image of $h$ intersects $U \in \cU$, and by the arbitrary choice of $U$ the same is true for all elements of $\cU$.
Thus by \ref{item:all_vertices_separate}, every vertex of $U^*$ must be contained in the image of $h$, which shows that $h$ is not a proper subgraph embedding.
\end{proof}

\begin{proof}[Proof of Theorem \ref{thm:counterexample_kappa-EPP_tree}]
Combine Theorems \ref{thm:not_proper_self-subtree} and \ref{thm:graph_without_kappa-EPP}.
\end{proof}

\section{Proof of the main theorems}\label{sec:proof_of_main_thm}

In this section we prove Theorems \ref{thm:intro_EPP}, \ref{thm:intro_aleph_0-EPP}, and \ref{thm:intro_kappa-EPP}.
Some of the definitions and lemmas that we need can be found in a similar or identical form in \cite{BEEGHPT22}, namely Definitions and Lemmas \ref{def:kappa-embeddable}--\ref{def:hordes}, Definition \ref{def:U-representation}, Lemma \ref{lem:U-representation}, and Lemma \ref{lem:extending_hordes}.
On the other hand, some of the more involved auxiliary results in this section are new. Lemma \ref{lem:regular_hordes} is a strengthening of \cite{BEEGHPT22}*{Lemma 7.3} and its proof requires new ideas, and Lemma \ref{lem:uncountable_hordes} is new.

\subsection{$\kappa$-closure}

In this subsection, we prove Lemma \ref{lem:kappa_closure} and give the definitions necessary to formulate it.
Let $G$ be a graph with a tree-decomposition $(T, \cV)$.

\begin{defn}[$G(U), \protect\overcirc{G}(U)$]
Given a subtree $U$ of $T$, we write $G(U)$ for the subgraph of $G$ induced by $\bigcup \{ V_u : u \in V(U)\}$.
If the unique $T$-minimal node of $U$ is a successor of a node $t$ of $T$, then we write $\overcirc{G}(U)$ for the subgraph of $G$ induced by $\bigcup \{ V_u : u \in V(U)\} \setminus V_t$.
If $U$ consists of a single vertex $u$, we also write $G(u)$ instead of $G(U)$.
\end{defn}

\begin{defn}[hinged]\label{def:hinged}
Let $t \in V(T)$ and $s, s' \in \successor_T(t)$.
We call a subgraph embedding $f: \overcirc{G}(\br_T(s)) \to \overcirc{G}(\br_T(s'))$ \emph{hinged (at $t$)} if any vertex $v$ of $\overset{\circ}{G}(\br_T(s))$ has the same neighbourhood in $G(t)$ as $f(v)$.
\end{defn}

\begin{defn}[$\kappa$-embeddable]\label{def:kappa-embeddable}
Let $(Q,\leq)$ be quasi-ordered. We say that an element $q \in Q$ is \emph{$\kappa$-embeddable} in $Q$ with respect to $\leq$ if there exist $\kappa$ distinct $q' \in Q$ with $q \leq q'$.
\end{defn}

\begin{lemma}[\cite{BEEGHPT22}]\label{lem:kappa_embeddable}
For any well-quasi-ordered set $(Q,\leq)$ and any infinite cardinal $\kappa$, the number of elements of $Q$ which are not $\kappa$-embeddable in $Q$ with respect to $\leq$ is less than $\kappa$.
\end{lemma}

\begin{defn}[$\kappa$-closed, $\kappa$-closure]\label{def:kappa-closure}
Let $U$ be a subtree of $T$ and $U'$ a subtree of $U$.
We say that $U'$ is \emph{$\kappa$-closed} in $U$ if for all $t \in U'$ and all $s \in \successor_U(t) \setminus V(U')$, the graph $\overcirc{G}(\br_T(s))$ is $\kappa$-embeddable in $\{ \overcirc{G}(\br_T(s')) : s' \in \successor_T(t)\}$ with respect to subgraph embeddings hinged at $t$.
Note that if $U'$ is $\kappa$-closed in $U$, then $U'$ is also $\kappa$-closed in $U''$ for any tree $U''$ with $U' \subseteq U'' \subseteq U$.

The \emph{$\kappa$-closure} of $U'$ in $U$ is the smallest subtree of $U$ that is $\kappa$-closed in $U$ and contains $U'$.
Such a subtree always exists since the set of all $\kappa$-closed subtrees of $U$ with $U'$ as a subtree contains $U$ and is closed under arbitrary intersections.
See Figure \ref{fig:example_kappa-closure} for an example.
\end{defn}

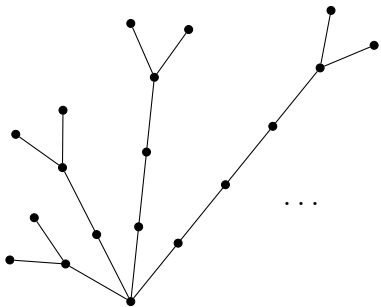
\begin{figure}[ht]
\center
\begin{tikzpicture}
\tikzstyle{v}=[draw,circle,color=black,fill=black,minimum size=3pt,inner sep=0pt]
\tikzstyle{u}=[draw,circle,color=white,fill=white,minimum size=2.5pt,inner sep=0pt]
\tikzstyle{r}=[draw,circle,color=red,fill=red,minimum size=2.5pt,inner sep=0pt]

\node[v] (g0) at (0,0) {};

\node[v] (g1) at (150:1) {};
\node[v] (g2) at (117:1) {};
\node[v] (g3) at (84:1) {};
\node[v] (g4) at (51:1) {};

\node[v] (g5) at (117:2) {};
\node[v] (g6) at (84:2) {};
\node[v] (g7) at (51:2) {};

\node[v] (g8) at (84:3) {};
\node[v] (g9) at (51:3) {};

\node[v] (g10) at (51:4) {};

\node[v] (h0) at (139:1.7) {};
\node[v] (h1) at (161:1.7) {};
\node[v] (h2) at (109.5:2.7) {};
\node[v] (h3) at (124.5:2.7) {};
\node[v] (h4) at (78:3.7) {};
\node[v] (h5) at (90:3.7) {};
\node[v] (h6) at (46.5:4.7) {};
\node[v] (h7) at (55.5:4.7) {};

\draw (g0) to (g1);
\draw (g0) to (g5);
\draw (g0) to (g8);
\draw (g0) to (g10);

\draw (g1) to (h0);
\draw (g1) to (h1);
\draw (g5) to (h2);
\draw (g5) to (h3);
\draw (g8) to (h4);
\draw (g8) to (h5);
\draw (g10) to (h6);
\draw (g10) to (h7);

\node at (2.3,1.3){$\dots$};

\end{tikzpicture}
\caption{Consider the tree $S$ from the figure rooted in its unique vertex of infinite degree with tree-decomposition $(S, \cW)$ where $W_s = \{ t \in V(S) : t \leq s \}$ for all $s \in V(S)$. Then the $\aleph_0$-closure of the root of $S$ in $S$ is $S$ itself.
However, when we delete all leaves of $S$, then $\aleph_0$-closure of the root of $S$ in $S$ is just the root of $S$.}
\label{fig:example_kappa-closure}
\end{figure}

\begin{lemma}\label{lem:kappa_closure}
Let $G$ be a graph with a tree-decomposition $(T, \cV)$ into finite parts such that the set of induced subgraphs of $G$ is lwqo.
Let $\kappa$ be an infinite cardinal, let $U$ be a subtree of $T$, and let $U'$ be a subtree of $U$ with $|U'| < \kappa$.
If either
\begin{itemize}
\item $\kappa = \aleph_0$ and $T$ is rayless, or
\item $\kappa$ is regular and uncountable,
\end{itemize}
then also the $\kappa$-closure of $U'$ in $U$ has size less than $\kappa$.
\end{lemma}

\begin{proof}
First of all, note that since the parts of $(T, \cV)$ are finite and the set of induced subgraphs of $G$ is lwqo, for every $t \in V(T)$ the set $\cS := \{ \overcirc{G}(\br_T(s)) : s \in \successor_T(t) \}$ is wqo by hinged subgraph embeddings.
Indeed, label the vertices of all graphs from $\cS$ with their neighbourhood in the finite set $V_t$.
Since every label-preserving subgraph embedding of one graph from $\cS$ into another is hinged and since the elements of $\cS$ are lwqo, they are also wqo by hinged subgraph embeddings.

Let $\overline{U'}$ denote the $\kappa$-closure of $U'$ in $U$.
Then every vertex $u$ of $\overline{U'}$ has degree less than $\kappa$.
Indeed, due to its minimality, $\overline{U'}$ contains precisely all successors $s$ of $u$ in $U$ for which $\overcirc{G}(\br_T(s))$ is not $\kappa$-embeddable in $\{\overcirc{G}(\br_T(s')) : s' \in \successor_T(u) \}$ with respect to hinged subgraph embeddings.
Since this set is lwqo, $u$ has less than $\kappa$ children in $\overline{U'}$ by Lemma \ref{lem:kappa_embeddable}.

Finally, we deduce that $|\overline{U'}| < \kappa$.
If $\kappa$ is regular and uncountable, then $|\overline{U'}| < \kappa$ since every tree on $\kappa$ nodes contains a vertex of degree $\kappa$.
Now suppose that $\kappa = \aleph_0$ and that $T$ is rayless.
Then $\overline{U'}$ is a rayless, locally finite graph and hence $|\overline{U'}| < \aleph_0$ by König's infinity lemma (see \cite{diestel2015book}*{Proposition 8.2.1}).
\end{proof}

\subsection{Hordes}

For notational convenience, we fix for this subsection:
\begin{itemize}
\item a graph $G$ such that the set of induced subgraphs of $G$ is lwqo, and
\item a tree-decomposition $(T,\cV)$ of $G$ into finite parts.
\end{itemize}

\begin{defn}[suitable]\label{def:suitable}
Let $U$ be a subtree of $T$ and let $\Gamma$ be any graph.
We call a subgraph embedding $f:G(U) \to \Gamma$ \emph{suitable at a node} $u \in U$ if there is a subgraph embedding $g: G(\br_T(u)) \to \Gamma$ such that $f$ and $g$ coincide on $G(u)$.
We say that $f$ is \emph{suitable} if $f$ is suitable at every $u \in U$.
\end{defn}

\begin{defn}[hordes]\label{def:hordes}
Let $U$ be a subtree of $T$ and let $\Gamma$ be any graph. We call $(f_i)_{i \in I}$ a \emph{$U$-horde} (in $\Gamma$) if:
\begin{itemize}
\item $f_i$ is a subgraph embedding of $G(U)$ into $\Gamma$ for all $i \in I$,
\item the images of $f_i$ for $i \in I$ are pairwise disjoint, and
\item $f_i$ is suitable for all $i \in I$.
\end{itemize}
Note that if $(f_i : i \in I)$ is a $T$-horde in $\Gamma$, then $\Gamma$ contains $|I|$ disjoint copies of $G$.
We say that a $U$-horde $(f_i : i \in I)$ \emph{extends} a $U'$-horde $(f'_i : i \in J)$ if $U'$ is a subtree of $U$ and $J \subseteq I$ and $f_i \upharpoonright G(U') = f'_i$ for all $i \in J$.
\end{defn}

An ordinal-indexed sequence $(A_\alpha : \alpha \leq \gamma)$ of sets (or trees) is called \emph{increasing} if $A_\alpha \subseteq A_\beta$ for all $\alpha < \beta \leq \gamma$, and \emph{continuous} if $A_\beta = \bigcup \{ A_\alpha : \alpha < \beta \}$ for all limit ordinals $\beta \leq \gamma$.

\begin{lemma}[regular hordes]\label{lem:regular_hordes}
Let $\kappa$ be an infinite cardinal such that one of the following holds:
\begin{itemize}
\item $\kappa = \aleph_0$ and $T$ is rayless, or
\item $\kappa$ is regular and uncountable.
\end{itemize}
Let $U$ be a rooted subtree of $T$ with $|U| \leq \kappa$ and let $U'$ be a $\kappa$-closed rooted subtree of $U$ with $|U'| < \kappa$.
Furthermore, let $I$ be any set with $|I| \leq \kappa$ and let $J \subseteq I$ with $|J| < \kappa$.
Let $\Gamma$ be any graph and suppose that one of the following holds:
\begin{enumerate}[label=(\arabic*)]
\item\label{item:Gamma-X} $\Gamma - X$ contains a copy of $G$ for all $X \subseteq V(\Gamma)$ with $|X| < \kappa$, or
\item\label{item:I=J} $I = J$.
\end{enumerate}
Then any $U'$-horde $(f'_i : i \in J)$ in $\Gamma$ can be extended to a $U$-horde $(f_i: i \in I)$ in $\Gamma$.
\end{lemma}

We will refer to this lemma as either Lemma \ref{lem:regular_hordes} \ref{item:Gamma-X} or Lemma \ref{lem:regular_hordes} \ref{item:I=J}, depending on which of the two conditions is used.
Note that Lemma \ref{lem:regular_hordes} \ref{item:Gamma-X} implies that if $|T| \leq \kappa$ for $T$ and $\kappa$ as in the lemma, then $G$ has the $\kappa$-EPP.
The more general statement of Lemma \ref{lem:regular_hordes} will be useful in the proofs of later lemmas.

\begin{proof}
Let $\mu := |V(U) \times I| \leq \kappa$ and enumerate $V(U) \times I = \{(u_\alpha, i_\alpha): \alpha < \mu\}$ in such a way that for all $\alpha \leq \mu$ and $i \in I$, the set $\{u_\beta : \beta < \alpha, i_\beta = i\}$ induces a rooted subtree $S_i^\alpha$ of $U$.
Let $U_i^\alpha$ be the $\kappa$-closure of $S_i^\alpha \cup U'$ in $U$ for all $i \in J$ and let $U_i^\alpha$ be the $\kappa$-closure of $S_i^\alpha$ in $U$ for all $i \in I \setminus J$.
Since $|S_i^\alpha| < \kappa$ and $|U'| < \kappa$, we have $|U_i^\alpha| < \kappa$ in both cases by Lemma \ref{lem:kappa_closure}.
Also note that for all $\alpha < \mu$, there are less than $\kappa$ many $i \in I$ for which $U_i^\alpha \neq \emptyset$.

First, we establish that the sequence $(U_i^\alpha : \alpha \leq \mu)$ is continuous for all $i \in I$, i.e.\ the tree ${U'}_i^\alpha := \bigcup \{ U_i^\beta : \beta < \alpha \}$ coincides with $U_i^\alpha$ for all limit ordinals $\alpha \leq \mu$.
We assume that $i \in I \setminus J$; the proof for $i \in J$ is similar.
To see that ${U'}_i^\alpha$ is $\kappa$-closed in $U$, consider nodes $t \in {U'}_i^\alpha$ and $s \in \successor_U(t) \setminus {U'}_i^\alpha$.
Then there is $\beta < \alpha$ such that $t \in U_i^\beta$.
Since $U_i^\beta$ is $\kappa$-closed and $s \in \successor_U(t) \setminus U_i^\beta$, the graph $\overcirc{G}(\br_T(s))$ is $\kappa$-embeddable in $\{ \overcirc{G}(\br_T(s')) : s' \in \successor_T(t)\}$ with respect to subgraph embeddings hinged at $t$.
This shows that also ${U'}_i^\alpha$ is $\kappa$-closed in $U$.
Next, we show that ${U'}_i^\alpha$ is the smallest $\kappa$-closed tree containing $S_i^\alpha$.
If there was a $\kappa$-closed tree $S$ containing $S_i^\alpha$ with ${U'}_i^\alpha \not \subseteq S$, then there would be a $\beta < \alpha$ such that $U_i^\beta \not \subseteq S$.
Since the intersection of $\kappa$-closed trees is again $\kappa$-closed, $S \cap U_i^\beta$ would be a $\kappa$-closed tree containing $S_i^\beta$, contradicting the minimality of $U_i^\beta$.
Thus ${U'}_i^\alpha$ is the smallest $\kappa$-closed tree containing $S_i^\alpha$, which proves that ${U'}_i^\alpha = U_i^\alpha$.

Now we recursively define subgraph embeddings $f_i^\alpha : G(U_i^\alpha) \to \Gamma$ for $i \in I$ and $\alpha \leq \mu$ such that:

\begin{itemize}
\item $f_i^\alpha \upharpoonright G(U_i^\beta) = f_i^\beta$ for all $i \in I$ and $\beta < \alpha \leq \mu$,
\item $f_i^\alpha \upharpoonright G(U') = f'_i$ for all $i \in J$ and $\alpha \leq \mu$,
\item the images of $f_i^\alpha$ and $f_j^\alpha$ are disjoint for all $i \neq j \in I$ and $\alpha \leq \mu$,
\item $f_i^\alpha$ is suitable for all $i \in I$ and $\alpha \leq \mu$.
\end{itemize}
Then $(f_i^\mu : i \in I)$ will be the desired $U$-horde extending $(f'_i : i \in J)$.

We have $S_i^0 = \emptyset$ and thus $U_i^0 = \emptyset$ for all $i \in I \setminus J$, and we have $U_i^0 = U'$ for all $i \in J$ since $U'$ is $\kappa$-closed.
We let $f_i^0$ be the empty map for all $i \in I \setminus J$ and define $f_i^0 := f'_i$ for all $i \in J$.
Now consider any $\alpha \leq \mu$ and suppose that we have already defined $f_i^\beta$ for all $\beta < \alpha$ and $i \in I$.
If $\alpha$ is a limit, then for all $i \in I$, let $f_i^\alpha: G(U_i^\alpha) \to \Gamma$ be the subgraph embedding extending $f_i^\beta$ for all $\beta < \alpha$, which exists and is unique since $U_i^\alpha = \bigcup \{ U_i^\beta : \beta < \alpha \}$. It is straightforward to check that $f_i^\alpha$ satisfies the four properties listed above.

Finally, suppose that $\alpha = \beta + 1$ is a successor and that $f_i^\beta$ has been defined for all $i \in I$.
For all $i \neq \l := i_\beta$, we have $S_i^\alpha = S_i^\beta$ and thus $U_i^\alpha = U_i^\beta$ and we define $f_i^\alpha := f_i^\beta$.
It is left to define $f_\l^\alpha$.

First suppose that $U_\l^\beta = \emptyset$.
Note that $U'$ is non-empty since it contains the root of $T$, so as
$U_\l^\beta = \emptyset$ we have $\l \in I \setminus J$.
In particular we have $I \neq J$ and thus we may assume that \ref{item:Gamma-X} holds.
Consider the set $X := \bigcup \{ f_i^\beta(G(U_i^\beta)) : i \in I \}$.
Since $|U_i^\beta| < \kappa$ for all $i \in I$, since the parts of $(T, \cV)$ are finite, and since there are less than $\kappa$ many $i \in I$ for which $U_i^\beta \neq \emptyset$, it follows from regularity of $\kappa$ that $|X| < \kappa$.
Therefore, by \ref{item:Gamma-X} there is a subgraph embedding $f: G \to \Gamma$ such that $f(G)$ avoids $X$.
We define $f_\l^\alpha := f \upharpoonright G(U_\l^\alpha)$, which clearly satisfies all prerequisites.

Now suppose that $U_\l^\beta \neq \emptyset$.
If $u_\beta \in U_\l^\beta$, then $U_\l^\alpha = U_\l^\beta$ and we define $f_\l^\alpha := f_\l^\beta$.
Otherwise, $u_\beta$ must be an immediate successor of a node $t$ of $S_\l^\beta \subseteq U_\l^\beta$.
Note that by minimality of the $\kappa$-closure, we have $U_\l^\alpha = U_\l^\beta \cup \br_{U_\l^\alpha}(u_\beta)$.
We set $f_\l^\alpha(v) := f_\l^\beta(v)$ for all $v \in G(U_\l^\beta)$, so it remains to define $f_\l^\alpha$ for the vertices of $\overcirc{G}(\br_{U_\l^\alpha}(u_\beta))$.

Since $f_\l^\beta$ is suitable, there is a subgraph embedding $g: G(\br_T(t)) \to \Gamma$ such that $g$ and $f_\l^\beta$ coincide on $G(t)$.
Moreover, since $U_\l^\beta$ is $\kappa$-closed, the graph $\overcirc{G}(\br_T(u_\beta))$ is $\kappa$-embeddable in $\{ \overcirc{G}(\br_T(s)) : s \in \successor_T(t)\}$ with respect to subgraph embeddings hinged at $t$.
This means that there exists a $\kappa$-sized set $B \subseteq \successor_T(t)$ such that for all $b \in B$ there is a hinged subgraph embedding $h_b: \overcirc{G}(\br_T(u_\beta)) \to \overcirc{G}(\br_T(b))$.
We want to find an element $b^* \in B$ such that the image of $g \circ h_{b^*}: \overcirc{G}(\br_T(u_\beta)) \to \Gamma$ is disjoint from the set $\bigcup \{ f_i^\beta(G(U_i^\beta)) : i \in I \}$, which has size less than $\kappa$ as we have seen before.
As the graphs $\overcirc{G}(\br_T(b))$ for $b \in B$ are pairwise disjoint (which follows from the definition of a tree-decomposition) and $|B| = \kappa$, we can indeed find a $b^* \in B$ which is as required.
We define $f_\l^\alpha(v) := g(h_{b^*}(v))$ for all $v \in \overcirc{G}(\br_{U_\l^\alpha}(u_\beta))$, which completes the construction.

To see that $f_\l^\alpha$ is a subgraph embedding, first note that $f_\l^\alpha(\overcirc{G}(\br_{U_\l^\alpha}(u_\beta))) = g(h_{b^*}(\overcirc{G}(\br_{U_\l^\alpha}(u_\beta))))$ is disjoint from $f_\l^\alpha(G(U_\l^\beta))$ because $f_\l^\alpha(G(U_\l^\beta)) = f_\l^\beta(G(U_\l^\beta)) \subseteq \bigcup \{ f_i^\beta(G(U_i^\beta)) : i \in I \}$.
So $f_\l^\alpha \upharpoonright G(U_\l^\beta)$ and $f_\l^\alpha \upharpoonright \overcirc{G}(\br_{U_\l^\alpha}(u_\beta))$ are subgraph embeddings with disjoint images.
Thus it remains to show that for every edge $vw$ of $G$ with $v \in G(U_\l^\beta)$ and $w \in \overcirc{G}(\br_{U_\l^\alpha}(u_\beta))$, there is a $f_\l^\alpha(v)$--$f_\l^\alpha(w)$ edge\footnote{In the proof of Theorem \ref{thm:main_thm_for_top_minors}, we instead need to find $f_\l^\alpha(v)$--$f_\l^\alpha(w)$ paths in $\Gamma$ such that their sets of inner vertices for distinct edges $vw$ are disjoint from each other and from $G(U_\l^\beta)$ and $\overcirc{G}(\br_T(u_\beta))$.
The arguments to achieve this are similar to the arguments used in this proof.
}
in $\Gamma$.
Since $v \in V_t$ by the definition of a tree-decomposition and since $h_{b*}$ is hinged at $t$, there is an edge in $G$ between $v$ and $h_{b^*}(w)$.
Since $g$ is a subgraph embedding, there is also an edge in $\Gamma$ between $g(v) = f_\l^\beta(v) = f_\l^\alpha(v)$ and $g(h_{b^*}(w)) = f_\l^\alpha(w)$.

To see that $f_\l^\alpha$ is suitable, consider any $u \in U_\l^\alpha$.
If $u \in U_\l^\beta$, then $f_\l^\alpha$ is suitable at $u$ since $f_\l^\beta$ is.
Otherwise, $u$ is contained in $\br_{U_\l^\alpha}(u_\beta)$ and suitability at $u$ is witnessed by $g \circ h_{b^*} \upharpoonright G(\br_T(u))$.
\end{proof}

To deal with trees of singular cardinality, we use the following definition and lemma from \cite{BEEGHPT22}:

\begin{defn}[$U$-representation]\label{def:U-representation}
Let $U$ be any tree and write $\eta := \cf(|U|)$. A \emph{$U$-representation} is an increasing continuous sequence $(U_\alpha : \alpha \leq \eta)$ of subtrees of $U$ such that:
\begin{itemize}
\item $U_\eta = U$,
\item $|U_\alpha| < |U|$ for all $\alpha < \eta$, and
\item $U_\alpha$ is $|U_\alpha|^+$-closed in $U$ for all $\alpha < \eta$.
\end{itemize}
We say that the $U$-representation \emph{extends} a subtree $U'$ of $U$ if $U' \subseteq U_0$.
\end{defn}

\begin{lemma}[\cite{BEEGHPT22} Lemma 7.5]\label{lem:U-representation}
For every tree $U$ of singular cardinality and every subtree $U'$ of $U$ with $|U'| < |U|$, there is a $U$-representation extending $U'$.
\end{lemma}

Let $\kappa$ be a regular uncountable cardinal and $\Gamma$ a graph such that $\Gamma - X$ contains a copy of $G$ for all $X \subseteq \Gamma$ of size less than $\kappa$.
In Lemma \ref{lem:regular_hordes} \ref{item:Gamma-X}, we have seen that for every set $I$ of size at most $\kappa$ and every rooted subtree $U$ of $T$ of size at most $\kappa$, there is a $U$-horde $(f_i : i \in I)$ in $\Gamma$.
In the following lemma, we show that the same is true for all (not necessarily regular) uncountable cardinals $\kappa$.
To make the inductive proof of the lemma work, we prove the following stronger statement:

\begin{lemma}[uncountable hordes]\label{lem:uncountable_hordes}
Let $\mu$ and $\kappa$ be uncountable cardinals such that $\mu$ is regular and $\mu \leq \kappa$.
Let $\Gamma$ be any graph such that $\Gamma - X$ contains a copy of $G$ for all $X \subseteq V(\Gamma)$ with $|X| < \kappa$.
Let $U$ be a rooted subtree of $T$ with $|U| \leq \kappa$ and let $U'$ be a $\mu$-closed rooted subtree of $U$ with $|U'| < \mu$.
Let $I$ be any set with $|I| \leq \kappa$ and let $J$ be a subset of $I$ with $|J| < \mu$.
Then any $U'$-horde $(f'_i : i \in J)$ in $\Gamma$ can be extended to a $U$-horde $(f_i : i \in I)$.
\end{lemma}

\begin{proof}
We prove the lemma by induction on $\max(|U|, |I|)$ and distinguish between four cases.

\begin{case}
$\max(|U|, |I|) \leq \mu$ (base case).
\end{case}

We can extend $(f'_i : i \in J)$ to a $U$-horde $(f_i : i \in I)$ by Lemma \ref{lem:regular_hordes} \ref{item:Gamma-X} (where we insert $\mu$ for the variable $\kappa$ from Lemma \ref{lem:regular_hordes})
since
\begin{itemize}
\item $\mu$ is regular and uncountable,
\item $U'$ is $\mu$-closed in $U$,
\item $|U'| < \mu$,
\item $|U|, |I| \leq \mu$,
\item $|J| < \mu$, and
\item $\Gamma - X$ contains a copy of $G$ for all $X \subseteq V(\Gamma)$ of size less than $\mu$ (as $\mu \leq \kappa$).
\end{itemize}

\begin{case}
$\max(|U|, |I|) > \mu$ and $\max(|U|, |I|)$ is regular.
\end{case}

Let $U''$ be $\max(|U|, |I|)$-closure of $U'$ in $U$.
Since $|U'| < \mu < \max(|U|, |I|)$ and $\max(|U|, |I|)$ is regular and uncountable, also $|U''| < \max(|U|, |I|)$ by Lemma \ref{lem:kappa_closure}.
Moreover, we have $|J| < \mu < \max(|U|, |I|)$ and thus $\max(|U''|, |J|) < \max(|U|, |I|)$.
Therefore, we can apply the induction hypothesis to extend $(f'_i : i \in J)$ to a $U''$-horde $(f''_i : i \in J)$.
We can further extend $(f''_i : i \in J)$ to a $U$-horde $(f_i : i \in I)$ by Lemma \ref{lem:regular_hordes} \ref{item:Gamma-X} (where we insert $\max(|U|, |I|)$ for the variable $\kappa$ from Lemma \ref{lem:regular_hordes}) since
\begin{itemize}
\item $\max(|U|, |I|)$ is regular and uncountable (as $\max(|U|, |I|) > \mu$),
\item $U''$ is $\max(|U|, |I|)$-closed in $U$,
\item $|U''| < \max(|U|, |I|)$,
\item $|J| < \max(|U|, |I|)$,
\item $|U|, |I| \leq \max(|U|, |I|)$, and
\item $\Gamma - X$ contains a copy of $G$ for all $X \subseteq V(\Gamma)$ of size less than $\max(|U|, |I|)$ (as $\max(|U|, |I|) \leq \kappa$).
\end{itemize}

\begin{case}
$\max(|U|, |I|) > \mu$, $\max(|U|, |I|)$ is singular, and $|U| < |I|$.
\end{case}

Write $\eta := \cf(|I|)$.
Since $|I| = \max(|U|, |I|)$ is singular and therefore a limit cardinal and since $|J| < \mu < \max(|U|, |I|) = |I|$, we can find an increasing continuous sequence $(I_\alpha : \alpha \leq \eta)$ of infinite subsets of $I$ such that:
\begin{itemize}
\item $|U| \leq |I_0|$,
\item $J \subseteq I_0$,
\item $I_\eta = I$, and
\item $|I_\alpha| < |I|$ for all $\alpha < \eta$.
\end{itemize}
We recursively construct $U$-hordes $(f_i : i \in I_\alpha)$ extending $(f'_i : i \in J)$ and each other.
To start the recursion, we extend $(f'_i : i \in J)$ to a $U$-horde $(f_i : i \in I_0)$ by the induction hypothesis, which is possible since $\max(|U|, |I_0|) = |I_0| < |I| = \max(|U|, |I|)$.
For successor steps, let $\alpha < \eta$ and suppose that the $U$-horde $(f_i : i \in I_\alpha)$ has already been defined.
We have $\max(U, I_{\alpha + 1}) = |I_{\alpha + 1}| < |I| = \max(|U|, |I|)$.
Therefore, we can extend $(f_i : i \in I_\alpha)$ to a $U$-horde $(f_i : i \in I_{\alpha + 1})$ by the induction hypothesis applied with $\mu = |I_\alpha|^+$ since
\begin{itemize}
\item $|I_\alpha|^+$ is regular and uncountable (as $I_\alpha$ is infinite),
\item $|I_\alpha|^+ \leq |I| \leq \kappa$,
\item $U$ is $|I_\alpha|^+$-closed in $U$, and
\item $|U| \leq |I_0| \leq |I_\alpha| < |I_\alpha|^+$.
\end{itemize}
If $\alpha \leq \eta$ is a limit, let $(f_i : i \in I_\alpha)$ be the unique $U$-horde that extends $(f_i : i \in I_\beta)$ for all $\beta < \alpha$, which exists by continuity of the sequence $(I_\alpha : \alpha \leq \eta)$.

\begin{case}
$\max(|U|, |I|) > \mu$, $\max(|U|, |I|)$ is singular, and $|U| \geq |I|$.
\end{case}

Write $\eta := \cf(|U|)$.
Since $|U'| < \mu < \max(|U|, |I|) = |U|$, there is a $U$-presentation $(U_\alpha : \alpha \leq \eta)$ extending $U'$ by Lemma \ref{lem:U-representation}.
By considering a subsequence of the $U$-representation, we assume without loss of generality that $U_0$ is infinite.
Since $|J| < \mu < \max(|U|, |I|) = |U|$, we may also assume that $|J| \leq |U_0|$.
Since $|I| \leq |U|$, we can find a (not necessarily strictly) increasing continuous sequence $(I_\alpha : \alpha \leq \eta)$ of subsets of $I$ such that
\begin{itemize}
\item $I_0 = J$,
\item $I_\eta = I$, and
\item $|I_\alpha| \leq |U_\alpha|$ for all $\alpha \leq \eta$.
\end{itemize}
We recursively construct $U_\alpha$-hordes $(f_i^\alpha : i \in I_\alpha)$ extending $(f'_i : i \in J)$ and each other.
To start the recursion, we extend $(f'_i : i \in J)$ to a $U_0$-horde $(f_i^0 : i \in I_0)$ by the induction hypothesis, which is possible since $\max(|U_0|, |I_0|) = |U_0| < |U| = \max(|U|, |I|)$.
For successor steps, let $\alpha < \eta$ and suppose that the $U_\alpha$-horde $(f_i^\alpha : i \in I_\alpha)$ has already been defined.
We have $\max(|U_{\alpha + 1}|, |I_{\alpha + 1}|) = |U_{\alpha + 1}| < |U| = \max(|U|, |I|)$.
Therefore, we can extend $(f_i^\alpha : i \in I_\alpha)$ to a $U_{\alpha + 1}$-horde $(f_i^{\alpha + 1} : i \in I_{\alpha + 1})$ by the induction hypothesis applied with $\mu = |U_\alpha|^+$ since
\begin{itemize}
\item $|U_\alpha|^+$ is regular and uncountable (as $|U_\alpha|$ is infinite),
\item $|U_\alpha|^+ \leq |U| \leq \kappa$,
\item $|I_{\alpha + 1}| \leq |U_{\alpha + 1}| \leq |U| \leq \kappa$,
\item $U_\alpha$ is $|U_\alpha|^+$-closed in $U$ and thus in $U_{\alpha + 1}$ by the definition of a $U$-representation,
\item $|U_\alpha| < |U_\alpha|^+$, and
\item $|I_\alpha| \leq |U_\alpha| < |U_\alpha|^+$.
\end{itemize}
If $\alpha \leq \eta$ is a limit, then let $(f_i^\alpha: i \in I_\alpha)$ be the unique $U_\alpha$-horde extending $(f_i^\beta : i \in I_\beta)$ for all $\beta < \alpha$, which exists by continuity of $(I_\alpha : \alpha \leq \eta)$ and $(U_\alpha : \alpha \leq \eta)$.
\end{proof}

Our final lemma allows us to extend certain $U$-hordes, which we will find using Lemma \ref{lem:regular_hordes} or Lemma \ref{lem:uncountable_hordes}, to $T$-hordes:

\begin{lemma}[for extending hordes]\label{lem:extending_hordes}
Let $\mu$ be an infinite cardinal such that one of the following holds:
\begin{itemize}
\item $\mu = \aleph_0$ and $T$ is rayless, or
\item $\mu$ is regular and uncountable.
\end{itemize}
Let $U$ be a rooted subtree of $T$, and $U'$ a $\mu$-closed rooted subtree of $U$ of size less than $\mu$.
Then any $U'$-horde $(f'_i : i \in I)$ with $|I| < \mu$ in a graph $\Gamma$ can be extended to a $U$-horde $(f_i : i \in I)$ in $\Gamma$.
\end{lemma}

\begin{proof}
By induction on $|U|$.

\setcounter{case}{0}
\begin{case}
$|U| \leq \mu$ (base case).
\end{case}

Then the assertion follows from Lemma \ref{lem:regular_hordes} \ref{item:I=J}.

\begin{case}
$|U| > \mu$ and $|U|$ is regular.
\end{case}

Then $|U'| < \mu < |U|$, so also the $|U|$-closure $U''$ of $U'$ in $U$ has size less than $|U|$ by Lemma~\ref{lem:kappa_closure}.
Hence by the induction hypothesis, we can extend $(f'_i : i \in I)$ to a $U''$-horde $(f''_i : i \in I)$ in $\Gamma$.
We can further extend $(f''_i : i \in I)$ to a $U$-horde $(f_i : i \in I)$ in $\Gamma$  by Lemma \ref{lem:regular_hordes} \ref{item:I=J} since
\begin{itemize}
\item $|U|$ is regular and uncountable,
\item $U''$ is $|U|$-closed in $U$,
\item $|U''| < |U|$, and
\item $|J| \leq |I| < \mu < |U|$.
\end{itemize}

\begin{case}
$|U| > \mu$ and $|U|$ is singular.
\end{case}

Write $\eta := \cf(|U|)$.
Since $|U'| < \mu < |U|$, there is a $U$-representation $(U_\alpha : \alpha \leq \eta)$ extending $U'$ by Lemma \ref{lem:U-representation}.
Without loss of generality, suppose that $\mu \leq |U_0|$.
We recursively construct $U_\alpha$-hordes $(f_i^\alpha : i \in I)$ for $\alpha < \eta$ extending $(f'_i : i \in I)$ and each other.
To start the recursion, we extend $(f'_i : i \in I)$ to a $U_0$-horde $(f_i^0 : i \in I)$ by the induction hypothesis.
For successor steps, let $\alpha < \eta$ and extend the $U_\alpha$-horde $(f_i^\alpha : i \in I)$ to a $U_{\alpha + 1}$-horde $(f_i^{\alpha + 1} : i \in I)$ by the induction hypothesis applied with $\mu = |U_\alpha|^+$. This is possible because
\begin{itemize}
\item $|U_\alpha|^+$ is regular and uncountable,
\item $U_\alpha$ is $|U_\alpha|^+$-closed in $U$ and thus in $U_{\alpha + 1}$ by the definition of a $U$-representation,
\item $|U_\alpha| < |U_\alpha|^+$, and
\item $|I| < \mu \leq |U_0| < |U_\alpha|^+$.
\end{itemize}
If $\alpha$ is a limit, then for all $i \in I$ let $f_i^\alpha : G(U_\alpha) \to \Gamma$ be the unique function extending $f_i^\beta$ for all $\beta < \alpha$, which completes the construction.
\end{proof}

\subsection{Deducing the main theorems and their corollaries}

\introkappaEPP*

\begin{proof}
Let $(T, \cV)$ be a tree-decomposition of $G$ into finite parts, let $\kappa$ be any uncountable cardinal, and let $\Gamma$ be a graph such that $\Gamma - X$ contains a copy of $G$ for all $X \subseteq V(\Gamma)$ of size less than $\kappa$.
Moreover, fix any set $I$ of size $\kappa$.
We show that $\Gamma$ contains a $T$-horde $(f_i : i \in I)$; then the graphs $f_i(G)$ for $i \in I$ are $\kappa$ disjoint copies of $G$ in $\Gamma$ and the proof is complete.

Let $U$ denote the $\kappa^+$-closure in $T$ of the root of $T$ and note that $|U| < \kappa^+$ by Lemma \ref{lem:kappa_closure}.
By Lemma \ref{lem:uncountable_hordes} (with $J = \emptyset$) there is a $U$-horde $(f'_i : i \in I)$ in $\Gamma$.
By Lemma \ref{lem:extending_hordes} (with $\mu = \kappa^+$) this horde can be extended to a $T$-horde $(f_i : i \in I)$ in $\Gamma$.
\end{proof}

\introalephnaughtEPP*

\begin{proof}
Let $\Gamma$ be a graph such that $\Gamma - X$ contains a copy of $G$ for all $X \subseteq V(\Gamma)$ of size less than $\aleph_0$.
Moreover, fix any set $I$ of size $\aleph_0$.
Since $G$ is rayless, by Lemma \ref{lem:rayless_tree-decomposition} there is a tree-decomposition $(T, \cV)$ of $G$ into finite parts such that $T$ is rayless.
We show that $\Gamma$ contains a $T$-horde $(f_i : i \in I)$; then the graphs $f_i(G)$ for $i \in I$ are $\aleph_0$ disjoint copies of $G$ in $\Gamma$ and the proof is complete.

Let $U$ denote the $\aleph_1$-closure in $T$ of the root of $T$ and note that $|U| < \aleph_1$ by Lemma \ref{lem:kappa_closure}.
By Lemma \ref{lem:regular_hordes} \ref{item:Gamma-X} (with $\kappa = \aleph_0$ and $J = \emptyset$) there is a $U$-horde $(f'_i : i \in I)$ in $\Gamma$.
By Lemma \ref{lem:extending_hordes} (with $\mu = \aleph_1$) this horde can be extended to a $T$-horde $(f_i : i \in I)$ in $\Gamma$.
\end{proof}

\introEPP*

\begin{proof}
By Lemma \ref{lem:rayless_tree-decomposition} there is a tree-decomposition $(T, \cV)$ of $G$ into finite parts such that $T$ is rayless.
Let $U$ denote the $\aleph_0$-closure in $T$ of the root of $T$ and note that $U$ is finite by Lemma~\ref{lem:kappa_closure}.
Consider the function $f : \NN \to \NN, k \mapsto |G(U)| \cdot (k - 1)$, let $k \in \NN$, and let $\Gamma$ be a graph such that $\Gamma - X$ contains a copy of $G$ for all $X \subseteq V(\Gamma)$ of size at most $f(k)$.
We show that $\Gamma$ contains a $T$-horde $(f_i : i \in \{ 1, \dots, k \})$; then $\Gamma$ contains $k$ disjoint copies of $G$ and the proof is complete.

We begin by finding a $U$-horde $(f'_i : i \in \{ 1, \dots, k \})$ in $\Gamma$.
We define $f'_i$ for $i \leq k$ recursively, so let $i \leq k$ and suppose that $f'_j : G(U) \to \Gamma$ has been defined for all $j < i$.
Since the set $X := \bigcup \{ f_j(G(U)) : j < i\}$ has size at most $f(k)$, there is a subgraph embedding $g: G \to \Gamma$ whose image avoids $X$.
We set $f'_i := g \upharpoonright G(U)$, completing the recursive construction.

By Lemma \ref{lem:extending_hordes} (with $\mu = \aleph_0$), we can extend the $U$-horde $(f'_i : i \in \{ 1, \dots, k \})$ to the desired $T$-horde $(f_i : i \in \{ 1, \dots, k \})$ in $\Gamma$.
\end{proof}

Finally, we deduce Corollaries \ref{cor:excluded_path_EPP} and \ref{cor:excluded_path_kappa-EPP} using the following result of Jia \cite{jia2015excluding}:

\begin{thm}\label{thm:excluded_paths_lwqo}
For every $n \in \NN$, the class of all graphs that do not contain a path of length $n$ as a subgraph is lwqo.
\end{thm}

\excludedpathEPP*

\begin{proof}
Since no (induced) subgraph of $G$ contains a path of length $n$ as a subgraph, the set of induced subgraphs of $G$ is lwqo by Theorem \ref{thm:excluded_paths_lwqo}.
Thus $G$ has the EPP by Theorem \ref{thm:intro_EPP}.
\end{proof}

\excludedpathkappaEPP*

\begin{proof}
Since no (induced) subgraph of $G$ contains a path of length $n$ as a subgraph, the set of induced subgraphs of $G$ is lwqo by Theorem \ref{thm:excluded_paths_lwqo}.
If $\kappa = \aleph_0$, then $G$ has the $\kappa$-EPP by Theorem~\ref{thm:intro_aleph_0-EPP}.
If $\kappa$ is uncountable, we apply Lemma \ref{lem:rayless_tree-decomposition} to find a tree-decomposition of $G$ into finite parts.
Then $G$ has the $\kappa$-EPP by Lemma \ref{thm:intro_kappa-EPP}.
\end{proof}

\section{Classes defined by topological minors}
\label{sec:top_minors}

A \emph{topological minor embedding} of a graph $H$ into a graph $G$ is a map $f$ with domain $V(H) \cup E(H)$ that maps every vertex of $H$ to a vertex of $G$ and every edge of $H$ to a path in $G$ such that:
\begin{itemize}
\item $f \upharpoonright V(H)$ is injective,
\item $f(vw)$ is an $f(v)$--$f(w)$ path in $G$ for every edge $vw \in E(H)$,
\item for all $e \in E(H)$, the set of inner vertices of the path $f(e)$ does not contain any vertex $f(v)$ for $v \in V(H)$ or any inner vertex of a path $f(e')$ for $e \neq e' \in E(H)$.
\end{itemize}

Let $L$ be any finite set and suppose that the vertices of the graphs $H$ and $G$ are labelled by elements of $L$.
We say that $f$ is \emph{label-preserving} if every vertex $v$ of $H$ has the same label as $f(v)$.
Furthermore, we say that a class $\cG$ of graphs is \emph{lwqo by topological minors}, if for every finite set $L$, the class of all graphs from $\cG$ with vertices labelled by elements of $L$ is wqo by label-preserving topological minor embeddings.

With this definition, we can state the following version of Theorems \ref{thm:intro_aleph_0-EPP} and \ref{thm:intro_kappa-EPP} for topological minors:

\begin{thm}\label{thm:main_thm_for_top_minors}
Let $G$ be any graph such that the set of induced subgraphs of $G$ is lwqo by topological minors.
\begin{itemize}
\item If $G$ is rayless, then $\cT(G)$ has the $\aleph_0$-EPP.
\item If $G$ admits a tree-decomposition into finite parts, then $\cT(G)$ has the $\kappa$-EPP for every uncountable cardinal $\kappa$.
\end{itemize}
\end{thm}

\begin{proof}[Proof outline]
The proof is essentially the same as the proof of Theorems \ref{thm:intro_aleph_0-EPP} and \ref{thm:intro_kappa-EPP} given in Section \ref{sec:proof_of_main_thm}, with the difference that all subgraph embeddings must be replaced by topological minor embeddings throughout Section \ref{sec:proof_of_main_thm}.
Additionally, some statements and some notations concerning subgraph embeddings must be transferred to the setting of topological minor embeddings in a natural way.
The notational changes include:
\begin{itemize}
\item Let $f$ be a topological minor embedding of a graph $G$ into a graph $H$ and let $G'$ be a subgraph of $G$.
We write $f(G')$ for the subgraph of $H$ consisting of all vertices $f(v)$ for $v \in V(G')$ and all paths $f(e)$ for $e \in E(G')$.
The \emph{image} of $f$ is the subgraph $f(G)$ of $H$.
\item Let $f$ be a topological minor embedding of a graph $G$ into a graph $H$ and $f'$ a topological minor embedding of $H$ into a graph $I$. The \emph{concatenation} $f' \circ f$ is the topological minor embedding $g$ of $G$ into $I$ such that:
\begin{itemize}
\item $g(v) = f'(f(v))$ for all $v \in V(G)$, and
\item $g(e)$ for $e \in E(G)$ is the path in $I$ obtained by concatenating all paths $f'(e')$ for $e' \in E(f(e))$.\qedhere
\end{itemize}
\end{itemize}
\end{proof}

Finally, we deduce Corollary \ref{cor:TT_has_EPP}. We need the following theorem by Laver \cite{laver1978better}:

\begin{thm}\label{thm:trees_are_lwqo}
The class of all trees is lwqo by topological minors.
\end{thm}

\TThasEPP*

\begin{proof}
Combine Theorems \ref{thm:main_thm_for_top_minors} and \ref{thm:trees_are_lwqo}.
\end{proof}

\bibliographystyle{plain}
\bibliography{ref}

\end{document}